\documentclass{amsart}
\usepackage{epsfig,amssymb,graphicx,rotating,multirow,multicol}
\usepackage{amsmath,amsthm,amsfonts,amscd,relsize}
\usepackage[toc,page]{appendix}
\usepackage{caption,comment,marginnote,mathrsfs}
\usepackage[mathscr]{euscript}
\usepackage{epstopdf}
\usepackage{float}
\usepackage{hyperref}

\usepackage{blindtext}

\usepackage{epigraph}
\swapnumbers
\newtheorem{theorem}{Theorem}[subsection]

\newtheorem{lemma}[theorem]{Lemma}

\newtheorem{cor}[theorem]{Corollary}
\newtheorem{proposition}[theorem]{Proposition}
\theoremstyle{remark}
\newtheorem{remark}[theorem]{Remark}
\numberwithin{equation}{subsection}
\newcommand{\K}{\ensuremath{\mathbb{K}}}

\def\lay{\operatorname{\Cal L}}

\def\NN{\mathscr{N}}

{\catcode`\@=11 \gdef\mnote#1{\marginpar{\tiny
 \tolerance\@M\spaceskip2.6\p@ plus10\p@ minus.9\p@\rm#1}}}
\def\Dg:{\endgraf{\bf Dg:\enspace}\ignorespaces}

\let\Bbb\mathbb
\let\Cal\mathcal
\def\AA{\mathbb A}

\def\DD{\mathbb D}

\def\O{\mathop O}

\def\s{\sigma}
\def\sm{\smallsetminus}
\def\-{\sm}

\def\A{\mathscr{A}}
\def\B{\mathscr{B}}
\def\CC{\mathscr{C}}
\def\DDD{\mathscr{D}}

\newcommand{\be}{\begin{equation}}
\newcommand{\ee}{\end{equation}}
\let\ge\geqslant 
\let\le\leqslant 
\let\la\langle
\let\ra\rangle
\let\til\widetilde
\def\Im{\operatorname{Im}}
\def\Map{\operatorname{Mod}}

\def\res{\operatorname{res}}

\def\Z{\Bbb Z}
\def\R{\Bbb R}
\def\C{\Bbb C}
\def\Q{\Bbb Q}

\def\T{\Bbb T}
\def\S{\Bbb S}
\def\E{\mathcal E}

\def\LL{\mathcal L}
\def\SS{\Sigma}
\def\O{\mathcal O}
\def\P{\Bbb P}

\def\W{\mathcal W}
\def\Rp#1{\Bbb{RP}^{#1}}

\def\X{\mathcal X}
\def\XX{\mathcal H}

\def\Pin{\operatorname{Pin}}
\def\pt{\operatorname{pt}}

\def\id{\operatorname{id}}

\def\Pic{\operatorname{Pic}}

\def\lay{\operatorname{\Cal L}}

\def\conj{\operatorname{conj}}

\def\dsum{\bot\!\!\!\bot}
\def\bhv{\operatorname{\varUpsilon}}

\def\om{\omega}
\let\+\dsum

\let\a=\alpha

\def\Pic{\operatorname{Pic}}
\def\inh{{\rm inh}}

\def\FS{\Theta^{(1)}}

\let\oo=\varnothing
\let\e=\varepsilon
\let\:=\colon

\makeatletter
\newcommand{\addresseshere}{%
  \enddoc@text\let\enddoc@text\relax
}
\makeatother

\title[Conic Bundle]{ On wall-crossing invariant combinations of Welschinger numbers for rational surfaces with a conic bundle}

\author[]{S.~Finashin, V.~Kharlamov}

\begin{document}

\begin{abstract}
We refine the real enumerative invariants insensitive to changes in the real structure introduced in our previous paper on
del Pezzo surfaces, and extend these refined invariants to other rational surfaces.
The construction involves a choice of an auxiliary conic bundle together with a $\Pin^-$-structure.
We prove that the resulting invariants are
independent of these auxiliary choices and establish recursive relations for their computation.
\end{abstract}

\maketitle


\setlength\epigraphwidth{.68\textwidth}
\epigraph{\dots the course of my own investigations and views makes me think that 
[this my] consideration may be of service to those persevering labourers (amongst whom I endeavour to class myself), who, occupied in the comparison of physical ideas with fundamental principles, and continually sustaining and aiding themselves by experiment and observation, delight to labour for the advance of natural knowledge, and strive to follow it into undiscovered regions.}{M.~Faraday 1857, On the Conservation of Force. \\}



\section{Introduction} 

\subsection{Objects and goals}
The goal of this paper is to show how the counting scheme for real rational curves that is insensible to  wall-crossing,
introduced in \cite{surgery-invariant} for del Pezzo surfaces, can be extended to arbitrary
rational surfaces 
equipped with an auxiliary conic bundle structure.

Throughout the paper we denote by $X$ a nonsingular complex rational surface
endowed with a real structure given by 
an antiholomorphic involution
$\conj_X \: X\to X.$ 
By a {\it real conic bundle} we mean a 
proper regular map $\pi\:X\to\P^1$ equivariant with respect to $\conj_X$ and the standard real structure $\conj_{\P^1}$ on $\P^1$,
whose generic fibers $\pi^{-1}(p)$, $p\in\P^1$,
are non-singular rational curves,  while the
singular fibers split into pairs of  $(-1)$-curves.
As is known (and explained in Section 2.1 below), there is a natural fiber-preserving $\conj$-invariant embedding 
$X\subset \P(\E)$
into the projectivization of a suitable 3-dimensional 
real vector bundle $\E\to \P^1$,
in which the fibers of $\pi$ 
become conics in the $\P^2$-fibers of $\P(\E)$.

The advantage of a conic bundle structure in $X$ is twofold.

On the one hand, it allows us to
 introduce, on the real locus 
 $X_\R={\rm Fix}(\conj),$
 certain  inherited $\Pin^-$-structures
 (see sections \ref{inh-even}--\ref{inh-odd} for precise definition).
As in our previous work \cite{surgery-invariant},
these structures provide a natural
sign correction to the Welschinger numbers \cite{Welsch}. 

On the other hand, in the del Pezzo case studied in \cite{surgery-invariant}, we summed
the modified Welschinger numbers
over the divisor classes $D$ 
having a fixed
intersection index 
$m=-KD\in \Z.$
Since $K^2>0$ in the del Pezzo case,
such a count involves only a finite number of divisor classes.
In contrast, if $K^2\le 0$, their number is infinite.
Another role played by conic bundles 
is therefore to make the count finite by restricting the sum to divisor classes that, in addition to having a fixed 
$m=-KD$,
have a fixed intersection index
with the divisor class of the fibers of the conic bundle (see Lemma \ref{finiteness}). 

\subsection{Counting scheme}
Since $X$ is assumed to be a rational surface, we have $\Pic(X)=H^2(X)$ and using Poincar\'e duality identify it with $H_2(X)$. 
We define the {\it bi-layer}, $\lay^{m,n}(X)\subset\Pic(X)= H_2(X)$, as the set of divisor classes $D$ with $-KD=m$ and $CD=n$, where $C$ is the divisor class of
the fibers of a fixed
conic bundle $\pi : X\to \P^1$. In the real case, we let $\lay^{m,n}_\R(X)= \lay^{m,n}(X)\cap\ker(1+\conj_*)$.

\begin{lemma}\label{finiteness}
For any 
conic bundle $X\to \P^1$ and any $m,n\in \Z$, the number of effective 
divisor classes $D$ of bi-level $(m,n)$ is finite.
\end{lemma}

\begin{proof}
Follows from \cite[Example. 1.4.31]{lazar}
and very ampleness of $-K+rC$ for any $r$ sufficiently large (see Proposition \ref{very-ample}).
\end{proof}

The main object of this paper is the following double sequence of integers
\begin{equation}\label{invrnt-definition}
\NN_{m,n}(X,\theta)= \sum_{\alpha\in \lay^{m,n}_\R(X)}
i^{\hat q(\alpha)-m^2}w_k(\alpha),\quad k\in\{0,1\}, \,k=m-1\,\text{mod}\,2
\end{equation}
where (in accordance with the definitions given in Section \ref{prelim-pin}):
\begin{itemize}
\item $\theta$ is a  $\Pin^-$-structure on $X_\R$ inherited from a $\Pin^-$-structure on $\P_\R(\Cal E)$, 
\item  $w_k(\alpha)$ is the {\it modified Welschinger number},
which differs from the original Welschinger number
by a factor
$(-1)^{g_a(\alpha)}$, with $g_a(\alpha)=\frac
{\alpha^2+\alpha  K}2-1$,
\item $\hat q=q\circ \bhv $ where $\bhv$ is the Viro homomorphism (see Section \ref{lifting})
and $q : H_1(X_\R;\Z/2) \to \Z/4$ is
the quadratic function defined by $\theta$ (see, for example, \cite{Kirby}).
\end{itemize}
If the number of singular fibers is $s$, we restrict our choice of $\theta$ to {\it monic $\Pin^-$-structures}, that is to those with $ q(w^*_1(X_\R))=1$.

\begin{remark} The above sign correction 
of 
original Welschinger numbers 
results in counting of
real rational curves $A\subset X$ 
with sign "-" if 
the holonomy $h(A_\R)\in \Z/4$ of the $\Pin^-$-structure
along $A_\R$ belongs to $\{0,1\}$, and  with sign "+" otherwise
(cf., \cite[§ 2.1]{findegt}).
\end{remark}

\subsection{Main results} 
Note that the number $s$ of singular fibers of a conic bundle $X\to \P^1$ is $\chi(X)-4= 8-K^2$,
so, parity of $s$ is the same as for $\chi(X)$ and, thus, for $\chi(X_\R)$.

 \begin{theorem}\label{main-main} 
 The double sequence $\NN_{m,n}(X,\theta)$ is the same for all real conic bundles
 $X\to\P^1$ 
 with given number $s\ge 2$ of singular fibers and $X_\R\ne\varnothing$. In particular, this double sequence does not depend  on a choice of
 an inherited $\Pin^-$-structure $\theta$ (monic, if $s$ is odd).
Furthermore,
the numbers  $\NN_{m,n}(X,\theta)$ vanish if 
$n\notin 2\Z_{\ge0}$. Besides, if $s$ is even, then they vanish also if $m\notin 2\Z_{>0}$.
\end{theorem}

This theorem
allows to forget about auxiliary conic bundle and a $\Pin^-$-structure and
attribute numbers $\NN_{m,n}$
to any rational surface $X$ with fixed $K^2\le6$. 
These numbers depend only on $K^2$,
and in  Section \ref{positivity} we prove their positivity, $\NN_{m,n}\ge0$.

In the next theorem, 
by $GW_X(\alpha)$ 
we mean the  genus-$0$ Gromov-Witten  invariant
of $\alpha \in H_2(X)$ and, thanks to Theorem \ref{main-main}, drop $\theta$ from notation  $\NN_{2m,2n}(X,\theta)$

\begin{theorem}\label{magic-theorem}
For any $m\in \Z_{\ge 1}, n\in\Z_{\ge 0}$, and any real conic bundle $X\to\P^1$ 
with $s\ge 4$ singular fibers and $X_\R\ne\varnothing $, we have
\begin{equation}\label{mag-formula}
\NN_{2m,2n}(X)= 2^{m-3}\sum_{\alpha\in \lay^{m,n} (X)} (e\alpha)^2 GW_X(\alpha),
\end{equation}
where 
$e$ is an arbitrary element of $\la K,C\ra^\perp\subset H_2(X)$ with $e^2=-2$.
\end{theorem}

Due to this theorem, 
for any real conic bundle $X$ with $s\ge 4$ singular fibers,  computating $\NN_{2m,2n}$ can be reduced
to a computation of the second moments of Gromov-Witten invariants, 
$$
GW^{[2]}_{m,n}:=\sum_{\alpha\in \lay^{m,n} (X)} (e\alpha)^2 GW(\alpha), \, e\in (K,C)^\perp, \, e^2=-2,
$$
(which are indeed independent on a choice of a root $e\in (K,C)^\perp$, see Proposition \ref{root-moments}).  To compute these moments we 
provide a pair of recursive relations for computing simultaneously 
$GW^{[2]}_{m,n}$ and 
$$GW_{m,n}:=\sum_{\alpha\in \lay^{m,n} (X)} GW(\alpha)
$$
starting from $GW_{m,0}$,
$GW_{1,n}$, and 
$GW^{[2]}_{1,n}$ as initial valuers, see  Proposition \ref{GWm0} and Theorems \ref{simple-recursion-any-d}, \ref{moment-recursion-any-d}.

Note that, due to Theorem \ref{main-main}, if the number of singular fibers is even, the numbers $\NN_{m,n}$ vanish unless both $m,n$ are even. Thus, in the case of even number of singular fibers, the 
cited recurrence relations
reduce the computation of the whole double sequence $\NN_{m,n}$ to a computation of 
sequences $GW_{m,0}$, $GW_{1, n}$, and $GW^{[2]}_{1,n}$.

In the case of odd number of singular fibers, the above recurrence relations are missing the numbers $\NN_{m,n}$ with odd $m$. Thus, to complete the recursion 
we provide an additional recurrence relation  (see Corollary \ref{N-recursion}), which allows us to reduce computing the double sequence $\NN_{m,n}$ to a computation of 
three sequences of initial values, $GW_{1,n}$,  $GW^{[2]}_{1,n}$, and $\NN_{1,n}$.

\subsection{The relation to previously defined invariants for del Pezzo surfaces}
To compare with the results in \cite{surgery-invariant}, we recall that the invariants $\NN_m$ introduced there
\footnote{Our notation for $\NN_m$ in \cite{surgery-invariant} 
is $N_{m,k}$, where $k\in\{0,1\}$ is determined by 
$k=m-1\, \text{\rm mod }\, 2$.} 
were defined following the same counting scheme  (\ref{invrnt-definition}). The only difference is that
the summation was taken over $\alpha\in \lay^{m}_\R(X):=\cup_{n\ge 0}\lay^{m,n}_\R(X)$ and that the $\Pin^-$-structure was chosen differently and is not associated with a conic bundle.
Our new invariants $\NN_{m,n}$ refine the previous $\NN_m$ in the following sense.

\begin{theorem}\label{refinement}
For any 
$m\ge1$, we have
$\NN_m=\sum_{n\ge0}\NN_{m,n}.$
\end{theorem}

\begin{remark}\label{without-dP}
In \cite{surgery-invariant} we explored
an enumerative (curve counting) meaning of the numbers $\NN_{m}$, and by this reason all the results 
were stated only for del Pezzo surfaces.
However, since the deformation classification of real del Pezzo surfaces coincides with that of real rational surfaces having $K^2>0$, all the results 
in \cite{surgery-invariant}
can be extended to the latter class of surfaces.
Note also that the cases $5\le K^2\le 6$ were excluded 
from consideration in \cite{surgery-invariant} because
they required a different approach to define
the $\Pin^-$-structures.
\end{remark}

\subsection{Plan of the paper}\label{plan}
In Section \ref{preliminaries}, we 
reduce the study of arbitrary conic bundles to 3 particular types of embedded ones (see Theorem \ref{exhaustive}).
In Section \ref{prelim-pin}, we introduce the notion of inherited $\Pin^-$-structures and establish some basic properties of our
counting scheme.
Section \ref{proof-main} is devoted to proving 
Theorems stated in the Introduction.
In Section \ref{sect-recursion}, we provide recursion relations for computing numbers $GW_{m,n}$, $GW_{m,n}^{[2]}$ and finally
$\NN_{m,n}$.
In Sections \ref{upto8} and \ref{nine}, we compute the initial values for the above recursion,
in the cases $K^2>0$ and $K^2=0$, respectively. 
 In addition, in Section \ref{nine} we provide a generating function for $\NN_{m,n}$ in the case $K^2=0$.  
 In Section \ref{concluding} we give additional recursion relations 
and present tables with some numerical values of our invariants.
Next, we discuss the case of real rational surfaces that do not admit any conic bundle structure.
And finally, we present a "rule of alternation" which describes explicitly the inherited $\Pin^-$-structures $\theta$.

\subsection{Conventions}
Occasionally, we extend our considerations to
complex-analytic (rather than algebraic) varieties $Z$ (usually, to an open subset $Z\subset X$ in a rational surface).
We call such variety $Z$ {\it real} if it is endowed with 
an anti-holomorphic involution $\conj :Z\to Z$, called a {\it complex conjugation}, or {\it real structure}.
We denote by $Z_\R\subset Z$ the fixed point set of $\conj$ and call it  the {\it real locus} of $Z$.

A notion of conic bundles 
extends
literally to the complex-analytic setting, with any 
smooth complex-analytic space
$B$ as a base.
Such a conic bundle $\pi:Z\to B$ is called {\it real}, if $X$ and $B$ are real and $\pi$ is equivariant.
By {\it minimality} of a conic bundle $X\to\P^1$ we mean relative minimality, i.e., absence of exceptional curves in the fibers.

Given a pair of fiber bundles over the same base, by a {\it bundle isomorphism} we mean 
an isomorphism 
acting identically on the base.
By
a {\it bundle automorphism} we understand an automorphism sending
each fiber to itself.

Given a real Lefschetz family 
$f : \X\to \DD$ with nonsingular $\X$ and a nodal degeneration at $0\in\DD$ (which we will call Morse-Lefschetz families, see details in Section \ref{ML-families}) we mean by a
{\it real vanishing cycle of $X(z)=f^{-1}(z), z\ne 0,$ in $\X$}
a $\conj$-invariant representative of classical Lefschetz vanishing spheres.

\subsection{Acknowledgements} 
This research was initiated in 2024
during our Research in Residence visits to Centre International de Rencontres Math\'ematiques, Luminy, and to  Institut des Hautes \'Etudes Scientifiques, Paris. We completed it during our Research Fellows staying in MFO, Oberwolfach, 2026.
We thank these institutions for hospitality and excellent working conditions.

\section{Conic bundles}\label{preliminaries}

\subsection{Embedded conic bundles}\label{CanEmbed}
Let $\pi_{\Cal E} : \Cal E\to\P^1$ be a rank 3 vector bundle, $\P(\Cal E)$ its projectivization
and $X\subset \P(\Cal E)$ a nonsingular hypersurface such
that every fiber $X_p= X\cap \pi_\E^{-1}(p)$, $p\in \P^1$, is a reduced conic.
Then, the embedding $X\subset \P(\Cal E)$  enriched by the restriction $\pi:X\to\P^1$ of $\pi_{\E}$ is called an {\it embedded conic bundle}.
For any embedded conic bundle $\pi:X\subset \P(\E ) \to\P^1$, at every point $p\in \P^1$, one can find a neighborhood $p\in U\subset \P^1$
 over which
the conic bundle is isomorphic to one of two standard conic bundles
embedded into $U\times \P^2\cong \pi^{-1}(U)$ : 
one being given by equation $ax^2+by^2+cz^2=0$ where $x,y,z$ are homogeneous coordinates in $\P^1$ and $a,b,c$ 
are elements of the ring of regular functions $\Cal O(U)$  not vanishing at $p$, if 
the fiber over $p$ is nonsingular, and, otherwise, given by equation $ax^2+by^2+ct z^2=0$ where, in addition, $t$ is a local coordinate in $U$ 
vanishing at $p$ ({\it cf.} \cite[Lemme 1.5.2]{Beau}). Over $\C$
up to a choice of complex analytic isomorphism
 $U\times \P^2\cong \pi^{-1}(U)$ 
 one can choose $a, b, c$ equal to $1$, while over $\R$ up to a $\conj$-equivariant isomorphism they can be choosen equal to $\pm 1$.

On the other hand, as is well known 
\cite{Beau, iskov, sarkis},
for any conic bundle $\pi:X\to B$, there exists a natural, functorial, construction realizing it as an embedded conic bundle. Namely, the anticanonical line bundle $-K_X$  is
{\it relatively ample} and defines 
a so called {\it canonical embedding}  $\phi: X\hookrightarrow  \P(\Cal E)$ where $\Cal E=\pi_*(\Cal O(-K_X))$ is a 3-dimensional vector bundle.
The image $\phi(X)$ is the zero-locus for
a nonzero section of the line bundle $\O_{\P(\Cal E)}(2)\otimes \pi^*(\det \Cal E\otimes K_B)$.
For each point $p\in B$, the image $\phi(X_p)$ of the fiber $X_p=\pi^{-1}(p)$
is a conic in $\P(\Cal E_p)$.
This construction works for conic bundles over an algebraic, or complex-analytic, manfiold $B$ of any dimension. In particular, this implies the following statement.
\begin{proposition}\label{deform-canonical}
Any smooth family of conic bundles $\pi(t): X(t)\to \P^1, t\in S$, over a parameter space $S$ of any dimension, defines a smooth family of
3-dimensional vector bundles $\E(t)=\pi_*(\Cal O(-K_{X_t}))$ endowed with the
canonical
conic bundle embeddings 
$\phi(t): X(t)\hookrightarrow  P(\E(t)), t\in S$.
For real families of conic bundles, the families $\E(t), \phi(t)$ are real as well.
\qed
\end{proposition}

\begin{proposition}\label{unicity-general} Let 
$X\subset \P(\Cal E^X)$, $Y\subset \P(\Cal E^Y)$
be a pair of
complex or real
 conic bundles embedded into $\P^2$-bundles
$\P(\E^X)\to \P^1$, $\P(\E^Y)\to \P^1$.
Then, 
any 
conic-bundle
isomorphism
$f: X\to Y$ 
extends to 
a unique  bundle 
isomorphism $\P(\Cal E^X)\to \P(\Cal E^Y)$.
If $X$ and $Y$ are real embedded bundles, then this isomorphism is equivariant.
\end{proposition}

\begin{proof}
The uniqueness part of the statement follows from applying to each 
pair of $\P^2$-fibers, $(\P(\Cal E^X_p), \P(\Cal E^Y_p))$,  $p\in \P^1$, the uniqueness of a projective transformation $\P(\Cal E^X_p)\to \P(\Cal E^Y_p)$
extending a given isomorphism between two given reduced
conics, $X_p\subset \P(\Cal E^X_p)$ and $Y_p\subset\P(\Cal E^Y_p)$. 
The same uniqueness property
together with gluing argument reduces a proof of the existence part to its local counter-part, that is to considering restrictions $X_U, Y_U$ of $X,Y$ to small bundle neighborhoods $\pi^{-1}(U)$, 
$p\in U\subset \P^1$, of the fibers $\pi^{-1}(p), p\in\P^1$.

To justify a local existence, we argue as follows.
Let $Z_U\to \P(\Cal H)$ be a canonical embedding of a local conic bundle isomorphic to $X_U$ and $Y_U$, and let $f_U: X_U\to Y_U$ 
is a given fiber-preserving isomorphism. Since the local model depends only on the type, singular or nonsingular, of the 
central fiber $\pi^{-1}(p)$,
there exist isomorphisms $(\Phi_U,\phi_U): (U\times\P^2, X_U)\to (\P(\Cal H),Z_U)$ and $(\Psi_U,\psi_U): (U\times\P^2, Y_U)\to (\P(\Cal H),Z_U)$.  By functoriality of canonical embeddings, the automorphism
$\psi_Uf_U\phi_U^{-1}: Z_U\to Z_U$ extends up to an ambient $H_U: \P(\Cal H)\to \P(\Cal H)$. Now, a required extension of $f_U$ is given by $\Psi_U^{-1}H_U\Phi_U$.
\end{proof}

 By Grothendieck-Birkhoff theorem, every vector bundle on $\P^1$ over any underlying field of characteristic $\ne 2$, splits into a direct sum of line bundles, which are unique up to permutation  (see \cite{splitting}). 
Accordingly, in what follows, we set 
$$\Cal E=\E_{a_0,a_1,a_2}=\bigoplus_{i=0}^2 \O_{\P^1}(a_i).$$
For working with equations in $\P(\E)$
 we use the standard toric presentation of $\P(\E)$
 as a $\C^*\times\C^*$-quotient of $(\C^2\sm \{0\})\times (\C^3\sm\{0\})$, which provides
 us, in addition to standard homogeneous coordinates $(u,v)$ coming from the first factor, with  weighted-homogeneous coordinates $(x_0,x_1, x_2)$ coming from the second one. Recall that the action
of $\C^*\times\C^*$ is given as follows (see, for example, \cite{reid}):
\be\label{homogeneity}
\begin{aligned}
(\lambda, 1) &: (u,v; x_0, x_1, x_2) \mapsto (\lambda u, \lambda v; \lambda^{-a_0}x_0, \lambda^{-a_1}x_1, \lambda^{-a_2}x_2),  \\
(1, \mu) &: (u,v; x_0, x_1, x_2) \mapsto (u, v; \mu x_0, \mu x_1, \mu x_2).  \\
\end{aligned}
\ee

Consider  divisor classes $M_i$, $i=0,1,2$, of the projective line subbundles in $\P(\Cal E)$ obtained from $\E$ by dropping summand $\O_{\P^1}(a_i)$
(in other words, the class of $\{x_i=0\}\subset\P(\E)$). Consider also class
$M \in \Pic \P(\Cal E)\cong H_2(\P(\Cal E))$
%
of the tautological line bundle $\O_{P(\Cal E)}(1)$, and the fiber class $F$.
Standard homology computations then yield the following.

\begin{proposition}\label{classM}$($cf. \cite{reid,frei}$)$
The classes $M$ and $F$ form together a basis of $\Pic \P(\Cal E)= \Z^2$. They satisfy the following:
\begin{enumerate}
\item $F^2=0, FM^2=1, M^3=\sum_{i=0}^2 a_i.$
\item $K_\E=-3M +(\sum^2_{i=0} a_i -2)F,$ 
\item $a_i F-M=M_i$ for each $i=0,1,2$.
\item Any embedded conic bundle $X\subset \P(\Cal E)$ has divisor class $X=2M+dF$ for some $d\in\Z$.
\qed \end{enumerate}
\end{proposition}

The integer $d$ in Proposition \ref{classM}(4) will be called the {\it degree} of $X$ in 
$\P(\E)$.

\begin{proposition}\label{very-ample}
For every conic bundle $\pi : X\to \P^1$ the divisor classes $-K_X+ rC$, where $C$ is the fiber divisor class, are very ample for any $r$ sufficiently large.
\end{proposition}

\begin{proof} Consider the conic bundle $\pi : X\to \P^1$ as embedded in $\P(\E_{a_0,a_1,a_2})\to \P^1$ with each $a_i>0$.
Then, as is known (see, f.e., \cite[Lemma 2.7]{reid}), the divisor class $M$ is very ample on $\E_{a_0,a_1,a_2}$,
and therefore the class $M_X$ induced by $M$ on $X$ is also very ample. On the other hand, as it follows by adjunction from Proposition \ref{classM},
$-K_X= 3M_X-(\sum a_i-2)C - (2M_X+ d C)=M_X-(d-2+\sum a_i)C$. Since $C$ is nef, we can conclude from above that $-K_X+rC$ is very ample
for every $r\ge d-2+\sum a_i$.

\end{proof}

\begin{proposition}\label{unicity-canonical}
Let $\pi : \Cal E\to  \P^1$ be a rank 3 vector bundle and $X\subset P(\Cal E)$ a nonsingular hypersurface such
that every fiber $X_b= X\cap \pi^{-1}(b)$, $b\in B$, is a conic. 
Then, $\pi\vert_X : X\to B$ is a conic bundle. Furthermore, there exists a line bundle $\Cal L$ over $B$ such that 
$X\subset \P(\Cal E)=\P(\Cal E\otimes \Cal L)$ is a canonical embedding up to 
a bundle automorphism of $\P(\Cal E)$. 
\end{proposition}

\begin{proof} For the first statement, see \cite[Proposition 3.6.1]{prokh-lectures}. Since two projective bundles $\P(\Cal E)$ and $\P(\Cal E')$ over the same base $B$ are isomorphic 
over $B$
if an and only
if there exists a line bundle $\Cal L$ such that $\E'=\E\otimes \Cal L$ (see, for example, \cite[Corollary 9.5]{E-H}),
the second statement follows from the first statement and Proposition \ref{unicity-general}.
\end{proof}

In the following Proposition, $f_n=f_n(u,v)$ stand for homogeneous polynomials of degree $n$ in coordinates $[u:v]$ of $\P^1$.

\begin{proposition}\label{number-fibers}
Let $X\to \P^1$ be a conic bundle. Then:
\begin{enumerate}\item
$X$ admits a conic bundle embedding $X\subset \P(\E)$,
$\E=\E_{0,a_1,a_2}$,
where $0\le a_1\le a_2$ are uniquely determined by $X$.
\item
$X$ is defined in $\P(\E)$ by equation
$\sum_{0\le i, j\le2}\, A_{ij} x_ix_j=0$, where
 $x_0,x_1,x_2$ are weighted-homogeneous coordinates in 
the fibers of $O_{\P^1}, O_{\P^1}(a_1), O_{\P^1}(a_2)$.
and  the matrix $\{A_{ij}\}$ has form
\begin{equation}\label{eq-matrix}
\begin{pmatrix}
f_{d}&f_{a_1+d}&f_{a_2+d} \\
f_{a_1+d}&f_{2a_1+d}&f_{a_1+a_2+d}\\
f_{a_2+d}&f_{a_1+a_2+d}&f_{2a_2+d}\\
\end{pmatrix},
\end{equation}
where $d$ is the degree of $X$ in $\P(\E)$.
\item
If a hypersurface
$X$ of  $\P(\E)$  defined by  equation 
of the above form is non-singular, then
$X\subset \P(\E)$ is an embedded conic bundle.
\item
The number, $s$,
of singular fibers of the conic bundle $X\to\P^1$ is equal to $3d+2(a_1+a_2)$.
\end{enumerate}
\end{proposition}

\begin{proof} Part (1) of
 Proposition \ref{number-fibers} is immediate from the last statement of Proposition \ref{unicity-canonical}. In (2), the existence and the form of the quadratic matrix equation (\ref{eq-matrix}) follows, for instance, from  toric presentation of 
$\P(\E_{0, a_1,a_2})$ as a $\C^*\times\C^*$-quotient of $(\C^2\sm 0)\times (\C^3\sm 0)$, see, for example, \cite{frei}.
Part (3) follows from Proposition \ref{unicity-canonical}.
Part (4)  follows from that the number of singular fibers is equal to the number of zeros
of $\det A $, that is $\deg(\det A)=3d+2(a_1+a_2)$. 
\end{proof}
%

\subsection{Exhaustive representation of real deformation classes}
In what follows we let $\E^s=\E_{0,a_1,a_2}$,
where
$(a_1,a_2)$ is $(0,0)$  if $s=3k$,
$(0,1)$  if $s=3k+2$, and
$(1,1)$  if $s=3k+4$, $k\ge0$,
 and let $\XX^s$ be the linear system $2M+dF$ in $\P(\E^s)$.
 
 \begin{theorem}\label{exhaustive}
Any real conic bundle $X\to\P^1$ with  $s\ge 2$ singular fibers is real deformation equivalent
to an embedded conic bundle from the linear system  $\XX^s$.
\end{theorem}

Before proving Theorem \ref{exhaustive}, let us recall the following  result going back to A.~Comessatti \cite{com} (see \cite{iskov} for a contemporary presentation).
\begin{theorem}\label{R-minimal} Each $\R$-minimal real rational conic bundle $X\to \P^1$ with $X_\R\ne \emptyset $ is either 
$(1)$ a conic bundle with $X_\R= mS^2$, $m\ge1$, with $2m$ 
 singular fibers which are all real, or
$(2)$ a geometrically ruled surface $\Sigma_k$, $k\ge 0$, 
with the standard real structure. 
In the latter case $X_\R$ is homeomorphic to 
a torus, $\T^2$, if $k$ is even, and to a
Klein bottle, $\K^2$, otherwise.
\qed
\end{theorem}

The R-minimal models of type (1) and (2)  in Theorem \ref{R-minimal} will be called respectively {\it spherical} and {\it aspherical}.

When proving Theorem \ref{exhaustive},
to control the deformation class we use the following criterium.

\begin{theorem}[\cite{degt-kh}]\label{R-min-deform}
$(1)$ $\R$-minimal spherical conic bundles
are real deformation equivalent if and only if they have the same number of singular fibers.

$(2)$ $\R$-minimal 
aspherical conic bundles,
$X^1$ and $X^2$, are real deformation equivalent if and only if $X^1_\R$ and  $X^2_\R$ are homeomorphic.
Since $X^i$, $i=1,2$, is some geometrically ruled surface with standard real structure, $\Sigma_{r_i}$,
homeomorphism 
$X^1_\R=X^2_\R$
is equivalent to 
 $r_1=r_2\operatorname{mod} 2$.\qed
\end{theorem}

We say that a real conic bundle  $X$ is of {\it spherical or aspherical minimal type} if its R-minimal model
(resulted from blowing down in $X$ some number of real 
and some number of conjugate pairs of imaginary line-components of singular fibers)
is spherical or aspherical. Note that a real conic bundle is of aspherical minimal type if, and only if, it has no any empty real fiber.

The  following statement is a straighforward consequence of Theorem \ref{R-min-deform}.
\begin{cor}\label{determine-defclass}
Assume that real conic bundles $X$, $Y$ with nonempty real parts have the same number of singular fibers. Then they are deformation equivalent if:
\begin{enumerate}
\item either, $X$ and $Y$ are of aspherical minimal type, having no real singular fibers and homeomorphic real loci
$X_\R\cong Y_\R$.
\item or,
$X$ and  $Y$ are of aspherical minimal type and have the same number $\ge1$ of real singular fibers.
\item or, $X$ and $Y$ are of spherical minimal type and 
have
homeomorphic
$X_\R\cong Y_\R$, where a homeomorphism respects (up to reversing)  the cyclic order of the components induced by the projections to $\P^1_\R$.
 \qed
\end{enumerate}
\end{cor}

\subsection{Scheme of proof of Theorem \ref{exhaustive}}
We prove Theorem \ref{exhaustive} analysing the three types of conic bundles in Corollary \ref{determine-defclass} and
presenting explicit equations for representatives of all deformation classes.

 An equation of the form $\A x_0^2+\B x_1^2+\CC x_2^2=0$, where $\A,\B,\CC$ are homogeneous polynomials in $(u,v)$
  will be called {\it diagonal}.
  Choosing an affine coordinate $t=\frac{u}v$, we consider below the  roots of $\A=\A(t)$, $\B=\B(t)$, $\CC=\CC(t)$.
According to Proposition \ref{number-fibers} we have the following.
 
 \begin{lemma}\label{diagonal}
A diagonal equation defines in $\P(\E^s)$ a hypersurface, $X$, in the linear system $\XX^s$ if $\A,\B,\CC$ have degrees
$d,d+2a_1,d+2a_2$, where $0\le a_1\le a_2\le 1$ and $s=3d+2(a_1+a_2)$. 
This hypersurface  is real if $\A,\B,\CC$ are real.
Moreover:
\begin{itemize}\item
$X$ is a conic bundle if and only if all roots of the product $\A\B\CC$ are distinct.
In this case, the $s$ singular fibers of $X$ are over the roots of $\A\B\CC$.
\item
 The real fibers $C_t\subset X$ have $C_{t\R}=\oo$ if and only if $\A(t),\B(t),\CC(t)$ are of the same sign.
 \qed
 \end{itemize}
 \end{lemma}
 

\subsection{The case of aspherical minimal type without real singular fibers}\label{i}
\begin{proposition}\label{construction1}
If $s\ge2$ is even, then the linear system $\XX^s$
contains embedded real conic bundles  $X^0$, $X^1$ of aspherical minimal type with $s$ singular fibers, all of which are imaginary, and with
$X^0_\R=\T^2$, $X_\R^1=\K^2$.
%
\end{proposition}

By Lemma \ref{diagonal}, the equation 
 $\A x_0^2+\B x_1^2+\CC x_2^2=0$, where $\A$, $\B$, $\CC$ are generic real polynomials without real roots of degrees $d$, $d+2a_1$, $d+2a_2$
respectively define a real conic bundle $X\subset\P(\E^s)$, $\E^s=\E_{0,a_1,a_2}$,
without real singular fibers.
If the signs of $\A,\B,\CC$ are not the same, it follows also that
$X_\R\ne\oo$ and, thus, it is either $\T^2$ or $\K^2$. 
The following Lemma shows how is it determined by the parity of $a_1$ and $a_2$.

\begin{lemma}\label{T-or-K}
$(1)$ If $\A<0$ and $\B,\CC> 0$, 
 then $X_\R=\T^2$ for even $a_1+a_2$  and $X_\R=\K^2$ for odd.

$(2)$ If $\B<0$ while $\A,\CC>0$,  then $X_\R=\T^2$ for even $a_2$  and $X_\R=\K^2$ for odd.

$(3)$ If $\CC<0$ while $\A,\B>0$,  then $X_\R=\T^2$ for even $a_1$  and $X_\R=\K^2$ for odd.
\end{lemma}

\begin{proof}
We will consider only case (1), since in other cases a proof is analogous.
The complement of $X_\R$ in $\E_\R$ has 2 connected components. The component defined by inequality $\A x_0^2+\B x_1^2+\CC x_2^2\le 0$
is fibered in 2-discs and contains as a deformation retract the section $x_0=1, x_1=x_2=0$. The normal bundle of this section is isomorphic to $O_\R(a_1)\oplus O_\R(a_2)$,
which is trivial if and only if $a_1+a_2$ is even.
\end{proof}

\begin{proof}[Proof of Proposition \ref{construction1}]
Lemma \ref{T-or-K} implies that if $(a_1,a_2)$ is $(0,1)$ or $(1,1)$, or equivalently, even $s\ge2$ not divisible by $6$,
we may construct $X\subset\P(\E^s)$ with either $X_\R=\T^2$ or $X_\R=\K^2$ by our choice.
For $(a_1,a_2)=(0,0)$, or equivalently, $s$ divisible by $6$,
a diagonal equation may give only $X_\R=\T^2$ and for construction $X$ with $X_\R=\K^2$ we use the following equation for $X\subset\P(\E^s)$
(cf., Lemma \ref{diagonal}):
$$
\A(u^2+2v^2)x_2^2+\B(ux_1-vx_0)^2=(\B+\epsilon \CC)(vx_1+ux_0)^2, \, 0<\epsilon\ll 1,
$$
where $\A, \B, \CC$ are generic homogeneous positive polynomials of the same degree $d-2\ge0$ in  $(u,v)$, $d=\frac{s}3$.
The component of $\E_\R\sm X_\R$ defined by inequality
$
\A(u^2+2v^2)x_2^2 +  \B(ux_1-vx_0)^2\le (\B+\epsilon \CC)(vx_1+ux_0)^2$
is fibered in 2-discs and has as a deformation retract the section $x_2=0, x_0=u, x_1=v$. Since the normal bundle of this section is non trivial, 
$X_\R$ is homeomorphic 
to the Klein bottle.

When $\epsilon=0$ the conic bundle has imaginary multiple degenerate fibers but its real part is nonsingular.
\end{proof}

\subsection{The case of aspherical minimal type with $\ge1$ real singular fibers}\label{aspherical-models}

\begin{proposition}\label{construction3} For any $s\ge 2$ and any $s_\R$ with $1\le s_\R\le s, s_\R=s\,{\rm\, mod\,2}$, the linear system 
 $\XX^s$ contains a conic bundle $X$ of aspherical minimal type with $s_\R$ real singular fibers.
\end{proposition}

\begin{proof}
We start with the case $s_\R\ge2$, in which $X$ can be presented by an equation in a diagonal form 
$\A x_0^2+\B x_1^2+ \CC x_2^2=0$ where $\A,\B,\CC$ are real homogeneous polynomials  in $u$ and $v$ of degree
$\deg \A=d\le\deg \B= 2a_1+d\le \deg \C=2 a_2+d$ (see Lemma \ref{diagonal}). 

Note that we can split $s_\R\ge2$ into a sum $s_\R=r^\A+r^\B+r^\CC$, where 
$$\begin{aligned}
&0\le r^\A\le \deg \A,\qquad 0\le r^\B\le \deg \B,\qquad 1\le r^\CC\le \deg \CC,\\
&r^\A, r^\B,r^\CC\ \text{ are of the same parity as } \deg \A=\deg \B=\deg \CC\,{\rm mod}\, 2.
\end{aligned}
$$

Then, we pick $\A$, $\B$ and $\CC$ with $r^\A$, $r^\B$ and $r^\CC$ real mutually distinct roots.

To obtain $X$ of aspherical minimal type we need that $C_\R\ne\oo$ for all conic-fibers $C$, which means that $\A$, $\B$ and $\CC$ will not be of the same sign
for all $t\in\R$.
This is achieved if $\A$, $\B$ and $\CC$ have respectively signs $-$, $+$ and $+$ at $-\infty$ and 
their real roots are displaced in an order encoded by  the following sequences of letters $\A$, $\B$, $\CC$ (the order of letters refers to the order of roots of these polynomials):

$\bullet$ { If  $s$ is even},
the order of roots is
$\B\dots \B\CC\A\dots \A\CC\dots \CC$,
where fragments $\A\dots \A$, $\B\dots \B$ and  $\CC\dots \CC$ are of lengths
$r^\A$, $r^\B$ and $r^\CC-1$ respectively.

$\bullet$ { If $s$ is odd}, the order of roots is 
$\CC\dots \CC\A\dots \A\B\dots \B$, where fragments $\A\dots \A$, $\B\dots \B$ and  $\CC\dots \CC$ are of lengths
$r^\A$, $r^\B$ and $r^\CC$ respectively.

It is left finally to construct $X$ with $s_\R=1$ (and in particular, with odd $s$).
Such $X$ can be defined by equation $F=\A x_0^2+2\DDD x_1x_2+\e(\B x_1^2+\CC x_2^2)$, $0<\e<\!<1$, 
where $\A,\B,\CC,\DDD$ are real polynomials of degrees $d$, $d+2a_1$, $d+2a_2$, $d+a_1+a_2$,
with mutually distinct roots,
which implies (see Lemma \ref{diagonal}) that $X$ is from the linear system $\XX^s$ if $s=3d+2(a_1+a_2)$, so, $d\ge 1$ is odd.
The real singular fibers of $X$ are over the roots of 
$$\det\left[\begin{matrix}\A&0&0\\0&\e \B&\DDD\\0&\DDD&\e \CC\end{matrix}\right]=\A(\e^2\B\CC-\DDD^2),$$
and if we choose $\A$ to have a unique real root and $\B\CC$ to take negative values at the roots of $\DDD$, then
$X$ will be an aspherical minimal type conic bundle with only one real singular fiber.
\end{proof}

\subsection{The case of spherical minimal type}\label{ii}
\begin{proposition}\label{spherical-construction}
  For any integers $s\ge s_\R\ge2$, such that $s_\R=s\,{\rm mod}\,2$
  and any integer partition $s_\R=s_1+\dots+s_k$, where $s_i\ge2$ for $1\le i\le k$,
there exists a conic bundle $X$ of spherical minimal type in the linear system $\XX^s$, with $s$ complex and $s_\R$ real singular fibers,
where  $s_1,\dots, s_k$ are the numbers of real singular fibers on the connected components of $X_\R$
listed in the cyclic order determined by $\pi_\R:X_\R\to\P^1_\R$.
%
\end{proposition}

We start proving with two elementary combinatorial lemmas and setting some terminology prior to them.
By a {\it balanced trisection} of integer $r\ge0$ we mean its partition $r=r^\A+r^\B+r^\CC$, into non-negative integers whose pairwise differences are allowed to be $-2,0,2$.
Clearly, for any $r\ge0$ such a trisection exists and unique up to permutation of the summands. Namely, it is 
\be\label{tri}
\begin{cases}
n+n+n &\text{ if } r=3n\\
(n-1)+(n+1)+(n+1)&\text{ if } r=3n+1\\
n+n+(n+2)&\text{ if } r=3n+2
\end{cases}
\ee
By a balanced trisection of a partition $r=r_1+\dots+r_k$, $r_i>0$ we mean a balanced trisection $r=r^\A+r^\B+r^\CC$ of $r$ and of each summand,
$r_i=r^\A_i+r^\B_i+r^\CC_i$, $1\le i\le k$, such that $r^\mathscr{Z}=r^\mathscr{Z}_1+\dots+r^\mathscr{Z}_k$ for $\mathscr{Z}=\A,\B,\CC$.

\begin{lemma}\label{trisection-partition}
Any partition $r=r_1+\dots+r_k$ with all $r_i\ge2$, admits a balanced trisection.
\end{lemma}

\begin{proof}
For proving by induction in $k$, 
it is enough to verify, as an induction step,  that for a balanced trisections $r=r^\A+r^\B+r^\CC$, and $r_k=r_k^\A+r_k^\B+r_k^\CC$
the differences $r^\A-r_k^\A$, $r^\B-r_k^\B$, $r^\CC-r_k^\CC$ is a trisection of $r-r_k$, if we order the summands so that 
$r^\A\le r^\B\le r^\CC$ and $r^\A_k\le r^\B_k\le r^\CC_k$. The verification is straightforward with the formulas \ref{tri} for trisections.
\end{proof}

A word $w$ formed by r letters in alphabet $\{\A,\B,\CC\}$ will be said to be {\it irreducible}
 if for every $1\le k<r$
the numbers of letters $\A$, $\B$ and $\CC$ in its substring of $k$ first letters are not of the same pairity.


\begin{lemma}\label{repeated}
For any balanced trisection $r=r^\A+r^\B+r^\CC\ge2$
 there exists an irreducible word  with $r^\A$ letters $\A$,  $r^\B$ letters $\B$ and  $r^\CC$ letters $\CC$.
\end{lemma}

\begin{proof}
Without loss of generality we may suppose that $r^\A\le r^\B\le r^\CC$, from where $r^\CC>0$ since $r>0$.

If $r$ is even and $r_\B\ge1$,
than the required word is 
$$\CC\A\dots \A\B\CC\dots \CC\B\dots \B\CC$$ where strings $\A\dots \A$, $\B\dots \B$, $\CC\dots \CC$, are of lengths $r^\A$, $r^\B-1$ and $r^\CC-2$ respectively. 
If $r_\B=0$, then our assumptions imply that $r_\A=2$, $r_\CC=0$, and the required word is simply $\A\A$.
If $r$ is odd, than the required word is 
$$\A\CC\dots \CC\B\dots \B\A\dots \A\CC$$ where strings $\A\dots \A$, $\B\dots \B$, $\CC\dots \CC$, are of lengths $r^\A-1$, $r^\B$ and $r^\CC-1$ respectively. 
\end{proof}

\begin{proof}[Proof of Proposition \ref{spherical-construction}]
Like before, we consider $X$ defined by an equation
in the diagonal form $\A x_0^2+\B x_1^2+\CC x_2^2=0$,
where $\A,\B,\CC$ are homogeneous polynomials in $(u,v)$ of degree $d,d+2a_1,d+2a_2$ respectively. 
The singular fibers of $X$ appear over the roots of $\A,\B,\CC$ in $\P^1$.
We will specify the order in which the real roots of $\A,\B,\CC$ are displaced on the line $\P^1_\R$, and displace the remaining imaginary
roots at arbitrary distinct complex-conjugate pairs of points of $\P^1\-\P^1_\R$, to make the polynomials real and without common roots.
Our assumptions allow to do it so that $\A,\B,\CC$ have the required degrees and thus, $X$ belongs to the linear system $\XX^s$.
To set the order of roots, we  apply Lemma \ref{trisection-partition}
to the partition $s_\R=s_1+\dots+s_k$, and obtain a balanced trisection with $s_i=s_i^\A+s_i^\B+s_i^\CC$, $i=1,\dots,k$.
The real roots of $\A\B\CC$ will be split into $k$ consecutive groups, and the
numbers $s_i^\A$, $s_i^\B$, $s_i^\CC$ will denote the number of 
roots of $\A$, $\B$ and $\CC$ in each group. The order of the latter ones are set to be the same as
the order of letters $\A,\B,\CC$  
in the corresponding $k$ words 
specified by Lemma \ref{repeated}, for each of the balanced trisections $s_i=s_i^\A+s_i^\B+s_i^\CC$.
Let us choose the signs of polynomials $\A,\B,\CC$ so that  at $-\infty$ the values 
are all positive.
Note that after passing $s_i$ roots of $\A,\B,\CC$ in each of the $k$ groups the signs of $\A,\B,\CC$ have to be the same, since
$s_i^\A,s_i^\B,s_i^\CC$ have the same parity. Now it is left to notice that the irreducibility condition for each of the $k$ groups of $s_i$ real roots
implies that the signs of $\A,\B,\CC$ will not be the same in the intervals between the $s_i$ roots of the same group, which implies that
the real loci of conic fibers withing each group will be non-empty and therefore, each of the $k$ groups will give precisely one connected component
of $X_\R$ with precisely $s_i$ real singular fibers.
\end{proof}

\subsection{Proof of Theorem \ref{exhaustive}} Using Theorems \ref{R-minimal}-\ref{R-min-deform} and Corollary \ref{determine-defclass}, it  is straightforward to check
that the examples constructed in Sections \ref{i}-\ref{ii} include representatives for each of the real deformation classes 
of real conic bundles with $s\ge 2$ singular fibers over $\C$ and nonempty real locus.
\qed

\subsection{Morse-Lefschetz families and wall-crossing}\label{ML-families}
We call a complex analytic family $X(z)$,  $z\in \mathbb D=\{z\in\C\,|\, \vert z\vert <1\}$
 a {\it Morse-Lefschetz family}, if:
\setlength{\leftmargini}{1.25em}
 \begin{itemize}
  \item
 $\X=\cup_z X(z)$ is a non-singular 3-fold, each $X(z)$ with $z\ne 0$ is non-singular, while
$X(0)$ has one singular point, a node $x_0\in X(0)$;
\item
the projection $f : \X\to \mathbb D$ is a proper map
being a submersion at every point except $x_0$;
\item
in appropriate local coordinates $z_1,z_2,z_3$ around $x_0$
 the projection can be written as $z= \sum_i z_i^2$.
\end{itemize} 
If $\X$ is equipped with a real structure sending $X(z)$ to $X(\bar z)$ for each $z\in \mathbb D$, the family is called {\it real}.
Then the nodal point of $X(0)$ has to be real and with respect to appropriate real local coordinates at this point the map $f $ is given by
\begin{equation}
f(z_1,z_2,z_3)=a_1z_1^2+a_2z_2^2+a_3z_3^2,
\quad a_1, a_2, a_3 \in {\mathbb R}, \quad
a_1a_2a_3\ne0.\label{ee11_new}\end{equation}

\begin{proposition}\label{embeddedMorseLefschetz} If a Morse-Lefschetz family  $\X=\cup_{z\in \mathbb D} X(z)$ is formed by 
hypersurfaces $X(z)\subset \P(\E)$  in the linear system $2M+dF$, $d\ge0$,
then for each $z\ne 0 $ the hypersurface $X(z)$ is a (nonsingular) conic bundle with respect to the projection
$\pi(z)$  induced from $\pi : \P(\E)\to \P^1$ and the vanishing class $e_z\in H_2(X(z))$ is orthogonal to $K_{X(z)}$
and $C_{X(z)}$,  where $C_{X(z)}$ stands for the class of the fibers of this conic bundle.
\end{proposition}

\begin{proof}
Follows from Propositions \ref{unicity-canonical} and \ref{classM}.
\end{proof}

\subsection{Real embedded conic bundles}

\begin{proposition}\label{spherical}
For any $s\ge4$
there exists a real Morse-Lefschetz family  $\X=\cup_{z\in \mathbb D} X(z)$
formed by hypersurfaces $X(z)$
 in the linear system $\XX^s$, such that 
 $X(z)$ are non-singular
(and thus, conic bundles) for 
$z\ne 0$ and 
$X(0)$  has one nodal singular point, and such that for $t\in\R\cap\mathbb D$
\begin{itemize}\item
for even $s$,
$X_\R(t)=\S^2$ if $t>0$, $X_\R(t)=\oo$ if $t<0$, and
$X_\R(0)=\{\pt\}$,
\item
for odd $s$,
$X_\R(t)= \Rp2\+\S^2$ if $t>0$, $X_\R(t)= \Rp2$, if $t<0$ and 
$X_\R(0)=\Rp2\+\{\pt\}$.
\end{itemize}
In both cases $\{\pt\}$ is a nodal point resulted from collapsing of $\S^2$-component.
\end{proposition}

\begin{proof}
A construction of $X=X(t)$, $t>0$ with the required topology of $X_{\R}$ is provided by
Proposition \ref{spherical-construction}, where we set $s_\R=2$ in the case of even $s$ and $s_\R=5$ in the case of odd.
In the first case, the two singular fibers must bound the only
connected component of $X_\R$, one fiber is minimum and another is maximum of the conic-bundle $X_\R\to\R\cong\P_\R^1\sm\{\infty\}$
(recall that the fiber, $C$, over $\infty$ is chosen to have $C_\R=\oo$).

In the second case, we consider a partition $s_\R=2+3$, corresponding to $X$ having $X_\R$ with 2 connected components,
one of which has 2 singular fibers which implies that this component is $\S^2$, like in the previous case,
and another has 3 singular fibers and thus, this component is $\Rp2$. 
To construct a Morse-Lefschets family $X(z), z\in\DD$,
we note that by the construction in Proposition \ref{spherical-construction},
the two singular fibers on $\S^2\subset X_\R$ lie over the roots of the same polynomial, $\A,\B$, or $\CC$ (without loss of generality, we may suppose it is $\A$).
Then we consider $X(t)$, $t>0$, defined by real polynomials
$\A_t$, and constant polynomials $\B_t=\B$, $\CC_t=\CC$, where only the two real roots of $\A$ are changing and are
moving towards each other and remain distinct until $t=0$, and become a pair of imaginary complex-conjugate roots for $t<0$.
\end{proof}

\begin{proposition}\label{connecting}
Any pair of real conic bundles in the 
linear system $\XX^s$ with $s\ge 4$
can be connected by a finite sequence of deformations and 
real Morse-Lefschetz families of conic bundles from this linear system.
If the real locus of the both conic bundles is non-empty the sequence of real Morse-Lefschetz families can be chosen not to 
contain conic bundles with empty real locus.

If the number of singular fibers is odd then the above deformations and Morse-Lefschetz families $X(t)$ can be enriched 
with a continuous family of real nonsingular fibers $C(t)\subset X(t)$ with $C_\R(t)\ne\emptyset$ placed away
from the nodes that
appear in the Morse-Lefschetz families.
\end{proposition}

\begin{proof} 
Recall that each point in the complement of the discriminant of the linear system  $2M+dF$ is a conic bundle (see Proposition \ref{number-fibers}(3)).

For $d>0$, this linear system is very ample (see, for example, \cite[Theorem 2.5]{reid}). Thus, the first statement follows from \cite[Theorem 2.5]{katz} (which states that, in characteristic zero, every generic projective line in 
a very ample linear system defines a Lefschetz pencil).
As for the second statement, it is sufficient to notice that the locus formed in a real  linear system by hypersurfaces with empty real locus is contractible.

Since $s=3d+2(a_1+a_2)$ and $0\le a_1\le a_2\le 1$, the assumption $s\ge 4$ implies that the only remaining case to be considered is that of $d=0$ and $\E^s=\E_{0,1,1}$.
In this case, the linear system $M$ provides a regular map to $\P^4$ and the image $Z$ of $\P(\E_{0,1,1})$ under this map
is a quadric cone over a nonsingular quadric in $\P^3$. Thus, in this case the discriminant in the linear system $2M+dF=2M$ of sections of $Z$ by quadrics has 2 irreducible components. One of them is the hyperplane in the space of
quadrics passing through the vertex of the cone, while the other is formed by quadrics  tangent to $Z$ at smooth points. Generic points of the second component correspond to 1-nodal sections (as it follows
from simple existence of such sections, easily justified by an example). Therefore, all but possibly one point in a generic subpencil of $2M$ correspond to a nonsingular or 1-nodal section, which implies 
the results stated.

For the final statement, we notice that for any pair of model (constructed in Sections \ref{i}, \ref{aspherical-models}, \ref{ii}) real conic bundles having the same odd number of singular 
fibers we can select the coefficients in the defining equations in such a way that the both bundles, $X$ and $Y$,
have a common real nonsingular fiber, $C=\pi_\E^{-1}(p)\cap X= \pi_\E^{-1}(p)\cap Y$,
with $C_\R\ne \oo$. Then, by a small real perturbation we transform $(X,C)$ and $(Y,C)$ into $(X', C_X)$ and $(Y',C_Y)$,
where $X', Y'$  are
placed in general position and $C_X=\pi_\E^{-1}(p)\cap X', C_Y= \pi_\E^{-1}(p)\cap Y'$.
Next, we observe that the interval $tX'+(1-t)Y'$, $t\in [0,1]$,
of the
Lefschetz pencil generated by $X', Y'$ 
can be
enriched with a family  of real fibers $tC_X+(1-t)C_Y=\pi_\E^{-1}(p)\cap (tX'+(1-t)Y')$,
which 
are all nonsingular, and hence placed away from the nodes of the pencil.
\end{proof}

\section{Combined signed count via $\Pin^-$-structures}\label{prelim-pin}

\subsection{Generalities on $\Pin^-$-structures on surfaces inherited from 3-folds}\label{pin-prelim} 
For a manifold $V$ we denote by $\Pin^-(V)$ the set of $\Pin^-$-structures on $V$.
If $\dim V\le 3$, then $\Pin^-(V)$
is non-empty and forms an affine space (torsor)  over $H^1(V;\Z/2)$. 
Recall
that for surfaces, $\SS$, there is a canonical 1--1 correspondence, $\theta\mapsto q_\theta$
 between $\Pin^-(\SS)$ and the set of quadratic functions $H_1(\SS;\Z/2)\to\Z/4$, such that $q_{\theta+h}=q_\theta+2h$, for $h\in H^1(\SS;\Z/2)$.
In particular,  $q_{\theta+w_1}=-q_\theta$ where $w_1=w_1(\SS)$.

\begin{lemma}\label{vanishPin}

Let $\Sigma$ be a closed surface which bounds a compact 3-manifold $V$,
and let $q: H_1(\Sigma;\Z/2)\to \Z/4$ be a quadratic function determined by a restriction to 
$\Sigma$ of a $\Pin^-$-structure defined on $V$.
Then
 $$
 q(h)=0, \text{ for any } h\in\ker(H_1(\Sigma;\Z/2)\to H_1(V;\Z/2)).
 $$
 \end{lemma}
 
 \begin{proof}\footnote{An intuitive geometric proof 
 can be given following 
 \cite[Proof of Lemma 2.4.2]{TwoKinds}.}
For any closed 2-manifold $\Sigma$ with a $\Pin^-$-structure represented by a quadratic function $q$ we have
(cf., \cite[Th. 1.20({\rm x})]{Brown})
$$2q(h)=Br(q)-Br(q+h^*)\in\Z/8,$$
where $Br$ is the Brown invariant, $h\in H_1(\Sigma;\Z/2)$ and $h^*\in H^1(\Sigma;\Z/2)$ is the dual class.
If $(\Sigma,q)$ bounds then $Br(q)=0$ since it is a cobordism invariant (see \cite[Lemma 3.6]{Kirby}).
And if $h$ bounds in $V$, then
$Br(q+h^*)=0$ for the same reason.
\end{proof}

\subsection{Lifting of the quadratic functions from $H_1(X_\R;\Z/2)$ to  $H_2^-(X)$}\label{lifting}
Recall that for a real algebraic surface $X$ with $X_\R\ne\oo$, there is
 {\it Viro homomorphism} $\bhv : H_2^-(X) \to H_1(X_\R;\Z/2)$ defined in 
 $H_2^-(X)=\ker (1+\conj_* : H_2(X)\to H_2(X))$
  (see \cite[Chapter 1]{DIK} or \cite{TwoKinds}).
This homomorphism respects the intersection form,
and induces an isomorphism
\[H_1(X_\R;\Z/2)\cong H_2^-(X)/(1-\conj_*)H_2(X).\]

This yields a correspondence $q_\theta\mapsto \hat q_\theta$ assigning to the
quadratic function $q_{\theta}: H_1(X_\R;\Z/2)\to\Z/4$ associated with 
$\theta\in\Pin^-(X_\R)$ the function
$$\begin{aligned}
&\hat q_\theta=q_\theta\circ \bhv\: H_2^-(X)\to \Z/4\,\quad\text{satisfying }\\
&\hat q_\theta(x+y)=\hat q_\theta(x)+\hat q_\theta(y)+2x\cdot y\ \rm{mod}\,4.
\end{aligned}$$

\subsection{Inherited $\Pin^-$-structures for $X$ having even number of singular fibers}\label{inh-even}
\begin{lemma}\label{two-sided}  
For any real embedded conic bundle $X\subset \P(\E)$
with even number of singular fibers, 
 $\P(\Cal E_\R)$ splits in a unique way into two halves (3-manifolds),
with $X_\R$ as their common boundary.
\end{lemma}

\begin{proof}
A 3-dimensional manifold bounding  
$X_\R$ is constructed via continuous
fiberwise filling of the real locus $C_{t\R}$ of conics
$C_t=\pi^{-1}(t)$, $t\in\P^1_\R$,  in the corresponding $\P^2_\R$-fibers of $\Cal E_\R$.
For a non-singular conic $C_{t\R}$, the complement $\P_{t\R}^2\sm C_{t\R}$ splits into a union $W_{t}^1\cup W_t^2$ of connected components.
A continuity of choice of one of these components,
$W_{t\R}^1$ means that the parity of $\chi(W_{t\R})$ changes if and only if $t$ overpass a critical value.
\end{proof}

\begin{remark}\label{equations-for-halves}
The halves in Lemma \ref{two-sided} are determined by the signs $F\ge0$ and $F\le0$
of a defining real equation $F=\sum A_{ij}(u,v) x_ix_j$
of $X$ given by Proposition \ref{number-fibers}. 
These signs are
well defined at each point of $X_\R$, 
since the 
homogeneity conditions in $\lambda$ and $\mu$ (see (\ref{homogeneity}) in Section \ref{CanEmbed}) imply pairity of the weighted degrees of the summands
$A_{ij}x_ix_j$.
\end{remark}

Fixing one half of this decomposition defines its boundary coorientation (in the outward normal direction).
This gives
a reduction map
$$\inh^X\:\Pin^-(\P(\Cal E_\R))\to\Pin^-(X_\R), 
$$
which is affine with respect to the inclusion homomorphism $H^1(\P(\E_\R);\Z/2)\to H^1(X_\R;\Z/2)$.
Reversing the coorientation changes  $\theta^X=\inh^X(\theta^\E)$ to
$\theta^X+w_1$ (cf., \cite[Lemma 1.10]{Kirby}).
In particular,
the quadratic function $\hat q(x)$ associated with $\theta^X$ is changed to
$$\hat q(x)+2w_1(x)=-\hat q(x)\,{\rm mod}\,4\quad.$$

We say that a $\Pin^-$-structure $\theta^X\in \Pin^-(X_\R)$ is {\it inherited} from $\P_\R(\E)$ if
$\theta^X=\inh^X(\theta^\E)$ for some choice of $\theta^\E$ and a half of $\P_\R(\E)$ (see Lemma \ref{two-sided}).
The set of all inherited $\Pin^-$-structures will be denoted
$\Pin^-_\inh(X_\R,\E)$, or simply $\Pin^-_\inh(X_\R)$
if $\E$ is clear from the context.

\begin{lemma}\label{torsor}
The set $\Pin^-_\inh(X_\R,\E)$ is a torsor with respect to the subgroup $\Z/2+\Z/2\subset H^1(X_\R;\Z/2)$
spanned by $w_1(X_\R)$ and class $[C_\R]^*$ dual to $[C_\R]$.
Furthermore,
for any pair $\theta_1,\theta_2\in \Pin^-_\inh(X_\R,\E)$, we have
$$\hat q_{\theta_1}(x)=\pm \hat q_{\theta_2}(x)+2\varkappa(C_\R\cdot x)\ \rm{mod}\,2, \quad \varkappa\in\Z/2.$$
\end{lemma}

\begin{proof}
The fibration $f_\E:\P(\E_\R)\to\P^1_\R$ 
leads to a short exact sequence 
$$
0\to H^1(\P^1_\R;\Z/2)\to H^1(\P(\E_\R);\Z/2)\to H^1(\P^2_\R;\Z/2)\to 0.
$$ 
It implies that $H^1(\P(\E_\R);\Z/2)$ is isomorphic to $\Z/2\oplus\Z/2$ and 
generated by the class dual to the fiber of $f_\E$ and $w_1(\P(\E_\R))$, because the latter one is mapped to $w_1(\P_\R^2)\ne0$.
%
The pull-back map  $H^1(\P(\E_\R);\Z/2)\to H^1(X_\R;\Z/2)$ 
sends 
the class dual to the fiber  to  $[C_\R]^*$ and, since by Lemma \ref{two-sided} $X_\R$ has a trivial  normal bundle in $\P(\Cal E_\R)$,
it sends $w_1(\P(\E_\R))$ to
 $w_1(X_\R)$.
Under reversal of coorientation the inherited $\Pin^-$-structures are shifted
by $w_1(X_\R)$, which in its turn leads to multiplication by $-1$ of 
the associated quadratic functions.
\end{proof}

\begin{proposition}\label{vanish-prop} 
For any $\theta\in\Pin^-_\inh(X_\R,\E)$, we have:
%
\begin{itemize}\item
$q_\theta([C_\R])=0$,
\item
$q_\theta(w_1^*)=0$ for the class dual to $w_1\in H^1(X_\R;\Z/2)$,
\item
$q_\theta([V_\R])=0$ for each real vanishing cycle $V\subset X$.
\end{itemize}
\end{proposition}

\begin{proof} For the dual of induced classes (like $[C_\R]$ and $w_1^*$)
it follows from Lemma \ref{vanishPin}. For the case of vanishing classes, see \cite[Lemma 2.4.2]{TwoKinds}.
\end{proof}

\subsection{Inherited $\Pin^-$-structures for $X$ having odd number of singular fibers}\label{inh-odd}
Consider an embedded real conic bundle $X\subset\P(\E)$ with an odd number $s$ of singular fibers.
Pick any real fiber $\P^2_t=\P(\E_{t})\subset\P(\E)$ for which the conic $C=X\cap \P^2_t$
is non-singular.
By the same reason as in the case of even number of singular fibers, we have the following.

\begin{lemma}\label{trivialization}
For any real embedded conic bundle $X\subset \P(\E)$
with odd number of singular fibers, and any choice of a distinguished nonsingular  fiber $C=X\cap \P^2_t$,
 $\P(\Cal E_\R)\-\P^2_{t\R}$ 
 splits in a unique way into two halves
with $X_\R\-C$ as their common boundary.
\qed
\end{lemma}


If the distinguished fiber $C$ corresponds to the point $(u=1, v=0)\in \P^1$, then the halves in Lemma \ref{trivialization} are determined by the signs $F\ge0$ and $F\le0$
of a defining real equation $F=\sum A_{ij}(u, 1) x_ix_j$ of $X$ given by Proposition \ref{number-fibers}
(cf. Remark  \ref{equations-for-halves}).

Lemma \ref{trivialization} provides a canonical (up to sign) coorientation of 
$X_\R \sm C_{\R}$ in $\P(\E_\R)$, which yields a reduction map
$$\inh_{C}^X\:\Pin^-(\P(\E_\R))\to \Pin^-(X_\R\sm C_\R).$$

\begin{lemma}\label{surjectivity-CB}
The restriction map $\res\: \Pin^-(X_\R)\to \Pin^-(X_\R\sm C_\R)$ is surjective.
\end{lemma}
\begin{proof}
Since $\Pin^-(X_\R)\ne\oo$, the restriction as a map of torsors is associated with  the inclusion homomorphism
$H^1(X_\R;\Z/2)\to H^1(X_\R\- C_\R;\Z/2)$. So, it remains to notice that the latter one is epimorphic.
\end{proof}

Any $\Pin^-$-structure 
$\theta\in\Pin^-(X_\R)$ 
whose restriction to $X_\R \sm C_{\R}$,
will be called
{\it inherited} from $\P(\E)$. The set of all such inherited structures will be denoted $\Pin^-_\inh(X_\R,C_\R,\E)$, or shortly, 
 $\Pin^-_\inh(X_\R,C_\R)$, if $\E$ is fixed.
 By Lemma \ref{surjectivity-CB}, the set $\Pin^-_\inh(X_\R,C_\R,\E)$ is non-empty.


\begin{lemma}\label{torsor-odd}
Under the assumptions of Lemma \ref{trivialization},
the set $\Pin_\inh^-(X_\R,C_\R,\E)$ is an affine space with respect to the subgroup of
$H^1(X_\R;\Z/2)$  spanned by $w_1(X_\R)$ and  $[C_\R]^*$. 
Furthermore, for any pair $\theta_1,\theta_2\in\Pin_\inh^-(X_\R,C_\R,\E)$, we have
$$q_{\theta_1}(x)=\pm q_{\theta_2}(x)+2\varkappa (C_\R\cdot x),\quad \text{ where }\varkappa\in\Z/2.$$
\end{lemma}

\begin{proof}
The proof is similar to that of Lemma \ref{torsor}. Indeed, the descents of $\Pin^-$-structures from $\P(\E_\R)$ to $X_\R\sm C_\R$
differ by $aw_1(X_\R\sm C_\R)+b [C'_\R]^*$ where $C'\subset X\sm C$ is a generic real fiber and $a,b\in \Z_2$. This implies that their lifts
up to $\Pin^-$-structures on $X_\R$ differ by combinations of $w_1(X_\R)$ and  $[C_\R]^*$. 
%
\end{proof}

\begin{proposition}\label{extension}
For any $\theta\in \Pin_\inh^-(X_\R,C_\R,\E)$, under 
the assumptions of Lemma \ref{trivialization} we have:
\begin{enumerate}\item
$q_\theta(C_\R)=0$, 
\item
$q_\theta(w_1^*)=\pm1$,
\item
$q_\theta(V_\R)=0$ for each real vanishing cycle 
$V$ of $X\sm C$ in $\P(\E)$.
\end{enumerate}
\end{proposition}

\begin{proof}
(1)--(2) are analogous to the proof of Proposition \ref{vanish-prop}.
Recall that $q_\theta(x)=x^2\,{\rm mod}\,2$ from where 
$q_\theta(w_1^*)$ have the same parity as $\chi(X_\R)$ and, therefore, as $s$.
\end{proof}

We say that $\theta\in \Pin_\inh^-(X_\R,C_\R,\E)$ is {\it monic}
if $q_\theta(w_1^*)=1$.
The set of monic structures form a subset denoted $\Pin_{\inh, {\rm mon}}^-(X_\R,C_\R,\E)\subset \Pin_\inh^-(X_\R,C_\R,\E)$.

\subsection{Welschinger numbers and their modification}\label{welsch-prelim}
{\it Traditional Welschinger numbers}, and our {\it modified Welschinger numbers} are 
defined respectively as
$$
\til w_k({\alpha})=\sum_{ [A]=\a} (-1)^{s_A},\quad \text{ and }\quad
w_k({\alpha})=\sum_{[A]=\alpha} (-1)^{c_A}=(-1)^{g_a(\a)}\til w_k(\a)
$$
for a symplectic 4-manifold $(X,\om)$ with an anti-symplectic involution $\conj:X\to X$ (called real structure)
and a generic tamed almost complex structure $J$  and anti-commuting with $\conj$.
In this definition, 
$\a\in H_2^-(X)=\ker \{1+\conj_* \:H_2(X)\to H_2(X)\}$, $c_1\cdot \alpha>0$
and the curves $A$ counted in the class $\a$
are real, rational
and pass through some generic $\conj$-invariant set of $m-1$ points
$\mathbf x\subset X$, $m=c_1\cdot\a$,
 including $k$ points in $X_\R$ 
 and $\frac12(m-k-1)$
pairs of complex conjugate imaginary points.
The integers $s_A$ and $c_A$
are respectively the number of solitary and cross-like real nodal points of $A$, while $g_a(\a)$ is the arithmetic genus, which
is congruent modulo $2$ to the total number of real nodes  $c_A+s_A$.
For a generic almost complex structure, all the counted curves are irreducible and reduced, see \cite{Welsch}.
For  invariance of $\til w_k({\alpha})$
under real deformations of $(X,\mathbf x)$, see also \cite{Welsch}.
For independence from a distribution of the $k$ real constraints among the connected components of $X_\R$,
see \cite{Br}.

We let  $\til w_k(\a), w_k(\a)$
be defined for all $k\in\Z_{\ge 0}$, $\a\in H^-_2(X)$, 
by setting 
them equal to zero
if the set of curves $A$ to be counted is empty, for instance, if any of the conditions $m>0$,
 $0\le k\le m-1$, and $k= m-1\, \rm{mod}\, 2$ is not satisfied.

In the algebraic setting, if $X$ is a real del Pezzo surface or $c_1\cdot \a>1$
and $X$ is a real rational surface with a generic complex structure, then, similar to Gromov-Witten numbers,  the numbers 
 $\til w_k(\a)$ and $w_k(\a)$ are enumerative, that is for generic point constraints, without passing to almost complex structure,  the counted curves are irreducible and reduced,
 and their count following the above definition gives the same result as for a generic almost complex structure, { \it cf.} \cite{GP}.

\begin{lemma}\label{welsch-vanish}
For any real rational symplectic 4-manifold $X$ 
(in particular, for any real rational conic bundle), 
$\tilde w_k(\alpha)=0$ if $\alpha^2\le -2$.
\end{lemma}

\begin{proof} 
For a generic 
almost complex structure $J$, no class $\alpha\in H_2(X)$ with $\alpha^2\le -2$ can be realized by a 
reduced irreducible $J$-holomorphic curve.
\end{proof}


\subsection{The numbers $\NN_{m,n}$}\label{defNN}
Consider a real embedded conic bundle $X\subset \E$ with $s\ge0$ singular fibers. Assume $X_\R\ne\varnothing $.
If $s$ is odd, choose a non-singular fiber $C\subset X$.
We pick
$\theta\in \Pin^-_{\inh}(X_\R,\E)$ if $s$ is even, and pick it in
$\Pin^-_{\inh,{\rm mon}}(X_\R,\E,C_\R)$ if $s$ is odd.
For any $m, n\in \Z$, we define
\begin{equation}
\NN_{m,n}(X,\theta)=\sum\limits_{\a\in\lay^{m,n}_\R(X)}i^{\hat q(\a)-m^2}w_k({\a}),\quad k\in\{0,1\}, \quad k=m-1\,\rm{mod}\, 2,\\
\end{equation}
where $
\hat q_\theta =q_\theta\circ \bhv  : H_2^-(X)\to \Z/4$ (see Section \ref{lifting}) and 
$$
\lay^{m,n}_\R(X)=\{ D\in H_2^-(X)\,|\, -KD=m, CD=n\}.
$$

For any real conic bundle $X\to \P^1$,
the above definitions can be applied to its canonical embedding $X\to \E$ (see Section \ref{CanEmbed}).
By Proposition \ref{unicity-general}, 
any other presentation of $X$ as an embedded conic bundle is ambient isomorphic to this one and so,
yields the same set
$ \Pin^-_{\inh}(X_\R,\E)$ if $s$ is even and 
 $\Pin^-_{\inh,{\rm mon}}(X_\R,\E,C_\R)$ (after fixing a conic-fiber $C$)
if $s$ is odd. As a result, we obtain the same collections
of numbers $\NN_{m,n}(X,\theta)$.
This justifies dropping
a specific embedding $X\subset \E$ from the notation $\NN_{m,n}(X,\theta)$.
(Note that for odd $s$, the definition of these numbers depends on the choice of $C$, since $\theta$ does depend.)

By Proposition \ref{deform-canonical}, every  real deformation family $X(z)\to \P^1$, $z\in\DD$, of real conic bundles
extends to a real deformation family of canonical embeddings $\phi(z): X(z)\hookrightarrow  \P(\E(z)), z\in \DD$. By equivariant version of Eheresmann theorem, 
for any real deformation family of real embedded conic bundles $X(z)\subset \P(\E(z))$,
there exist,
well-defined up to isotopy,  diffeomorphisms $\mu_{z_1,z_2} : (\E(z_1),X(z_1))\to (\E(z_2), X(z_2))$, $z_1,z_2\in\DD$, such that $\mu_{z_1,z_3}=\mu_{z_2,z_3}\circ \mu_{z_1,z_2}$
and $\conj \mu_{z_1,z_2} \conj = \mu_{\conj z_1,\conj z_2}$.
They induce well-defined {\it monodromy translation maps} 
$H_*(X(z_1))\to H_*(X(z_2))$, $H_*(\E(z_1))\to H_*(\E(z_2))$ for $z_1,z_2\in \DD$ at the level of homology, and $\Pin^-(X(t_1))\to \Pin^-(X(t_2))$, $\Pin^-(\E(t_1))\to \Pin^-(\E(t_2))$ for $t_1,t_2\in\DD_\R$ at the level of $\Pin^-$-structures,
which also we denote all by respectively $\mu_{z_1,z_2}$ and $\mu_{t_1,t_2}$.

\begin{proposition}\label{Pin-invariance}
Let $X(z)\subset \P(\E(z))$, $z\in\DD$, be a real deformation family of real embedded conic bundles having $s$ singular fibers. If $s$ is odd, let $C(z)\subset X(z)$, $z\in\DD$, 
be a family of nonsingular fibers such that
$C(z_2)=\mu_{z_1,z_2}C(z_1)$ and $\conj C(z)=C(\conj z)$.
Then, for any pair $t_1,t_2\in \DD_\R$
the  map $\mu_{t_1,t_2}$ gives a bijection
$$
\begin{cases}
\Pin^-_\inh(X(t_1)_\R,\E(t_1))\to \Pin^-_\inh(X(t_2)_\R,\E(t_2))\hskip-2mm&\text{if $s$ is even},\\
\Pin^-_{\inh,{\rm mon}}(X(t_1)_\R,C(t_1)_\R,\E(t_1))\to \Pin^-_{\inh,{\rm mon}}(X(t_2)_\R,C(t_2)_\R,\E(t_2))\hskip-2mm&\text{if $s$ is odd}.
\end{cases}
$$
Moreover, for any $\Pin^-$ structure $\theta_1$ sent by this bijection to $\theta_2$,
$$
\NN_{m,n}(X(t_1),\theta_1)=\NN_{m,n}(X(t_2), \theta_2), \ \text{ for all } m,n\in\Z.
$$
%
\end{proposition}

\begin{proof} These are straightforward consequences of definitions and deformation invariance of Welschinger numbers.
\end{proof}

\begin{cor}\label{wall-invariance}
Assume that for some 
$s\ge 2$
the double sequence $\NN_{m,n}(X,\theta)$ is independent 
of the choice of a real embedded conic bundle $X$ with $X_\R\ne\oo$
in the linear system $\XX^s$, and of the choice of
\begin{enumerate}\item
$\theta\in \Pin^-_\inh(X_{\R},\E^s)$, if $s$ is even,
\item 
a real nonsingular fiber $C\subset X$ and $\theta\in\Pin^-_{\inh,{\rm mon}}(X_{\R},C_\R,\E)$, if $s$ is odd.
\end{enumerate}
Then the double sequence $\NN_{m,n}(X,\theta)$  remains the same for all real 
conic bundles $X\to\P^1$ with $s$ singular fibers and $X_\R\ne\oo$.
\end{cor}

\begin{proof}
By Theorem \ref{exhaustive}, 
each real conic bundle $X$ 
having $s\ge 2$ singular fibers over $\C$ is deformation equivalent to a
real conic bundle $X'\in\XX^s$ embedded in $\P(\E^s)$. If $s$ is odd, this deformation can be enriched by a deformation of fixed fibers. According to Proposition \ref{deform-canonical} such a deformation induces a deformation $X(z)\hookrightarrow \P(\E(z))$ connecting the canonical embedding $X(t_0)=X\hookrightarrow  \P(\E(t_0))$ with the canonical embedding $X(t_1)=X'\hookrightarrow \P(\E(t_1))$
which in its turn determines a continuous family of $\Pin^-$-structures $\theta_t$ on $\P(\E(t))$ with $\theta_0=\theta$ as initial value.
By Proposition \ref{unicity-canonical} $X'\hookrightarrow \P(\E(t_1))$ is a conic bundle isomorphic to $X'\hookrightarrow \P(\E^s)$. Now, there remains to notice that the numbers  $\NN_{m,n}(X,\theta)$
are preserved under bundle deformations and bundle isomorphisms of $X(t)\hookrightarrow \P(\E(t))$.
\end{proof}

\section{Proof of main results}\label{proof-main}

\subsection{Coherence of inherited $\Pin^-$-structures under wall-crossing}\label{coherence}

Consider a pair of embedded real conic bundles $X_1, X_2\subset \P(\E)$
with an even number of singular fibers and assume that they are
related by a wall-crossing in the same linear system $2M+dF$, $d\in 2\Z_{\ge 0}$.
More formally, there is 
a real Morse-Lefschetz family $\pi : \Cal X\to \mathbb D$ (see Section \ref{ML-families})
formed by  $X(z)\in \vert 2M+dF\vert $, $z\in  \DD$, such that $X_1=X(t_1)$, $X_2=X(t_2)$ for some $t_1,t_2\in\DD_\R\sm\{0\}$.
%
Such a family is defined by a real equation $\sum A_{ij}(u,v,z) x_ix_j=0$ with coefficients depending analytically on $z$ and polynomial homogeneous in $u,v$ of degrees as in Proposition \ref{number-fibers}. This equation is
unique up to a factor $g(z)$ which is a $\conj$-equivariant function $g: \DD\to \C\-\{0\}$.
For each $t\in\DD_\R$, the signs of $\sum A_{ij}(u,v,t) x_ix_j$ are well-defined at each point of 
 $\P(\Cal E_\R)$, {\rm cf.}, Remark \ref{equations-for-halves}.
Therefore, a choice of family $\Cal X$
determines uniquely 
a continuous family of decompositions of 
 $\P(\Cal E_\R)$
into the halves $\sum A_{ij}(u,v,t) x_ix_j\ge 0$ and $\sum A_{ij}(u,v,t) x_ix_j\le 0$ having $X_\R(t)$ as a common boundary
({\it cf.,} definitions in Section \ref{inh-even}).

Accordingly, we say 
that  $\Pin^-$-structures $\theta^{X_k}\in \Pin^-_\inh(X_{k\R},\E)$, $k=1,2$
are {\it coherent} (with respect to the chosen family $\Cal X$), if there exists
a $\Pin^-$-structure $\theta\in\Pin^-(\P(\E_\R))$ such that $\theta^{X_k}=\inh^{X_k}(\theta)$, 
where $\inh^{X_k}$
 are determined by 
 the halves $\{\sum A_{ij}(u,v,t) x_ix_j\ge 0\}\subset \P(\E_\R)$ with $t=t_k$.

The following is an immediate consequence of definitions.
\begin{lemma}\label{exist-coher-ML-even}
For any real embedded conic bundle $X\subset \P(\E)$ belonging to a given real Morse-Lefschetz family of conic bundles with even number of singular fibers and any $\theta^X\in \Pin^-_\inh(X_\R,\E)$, there exists a  family of coherent inherited $\Pin^-$-structures on
this Morse-Lefschetz family that coincides with $\theta^X$ on $X_\R$.
\qed
\end{lemma}

Furthemore, for a small Milnor ball $B\subset \P(\E)$ around the nodal point of $X(0)$ 
the family $X(t)\sm B$
is locally trivial over a small concentric disc  $t\in \DD_\R^\e\subset \DD$, which yields
the {\it monodromy transport isomorphisms} 
$$\mu_t : H_2^-(X(0)\sm B)\to H_2^-(X(t)\sm B)= H_2^-(X(t)\sm E(t))=H_2^-(X(t))\cap [E(t)]^\perp,$$ where $E(t)\subset X(t)$ are the vanishing spheres
and for the isomorphism $H_2^-(X(t)\sm E(t))=H_2^-(X(t))\cap [E(t)]^\perp$
see, e.g.,  \cite[Prop. 2.3.2]{surgery-invariant}.
On the other hand, our
choice of the halves in $\P(\E_\R)$, introduced for $t\ne 0$,
extends continuously to $X_\R(0)\sm B$ and, thus, provides
a continuous family of inherited $\Pin^-$-structures $\theta^t\in\Pin^-(X_\R(t)\sm B)$ for all $t\in \DD^\e_\R$, which for $t\ne 0$ coincide with the restriction of the former $\theta^{X(t)}\in \Pin^-_\inh(X_\R(t),\E)$ to $X_{\R}(t)\sm B$.

\begin{proposition}\label{equality-coher-ML-even}
If $\theta^{X(t_1)}\in \Pin^-_\inh(X_\R(t_1),\E) $, $\theta^{X(t_2)}\in \Pin^-_\inh(X_\R(t_2),\E) $ with $t_1,t_2\in\DD_\R\sm\{0\}$ are coherent 
with respect to a chosen real Morse-Lefschetz family $\pi : \Cal X\to \mathbb D$
of conic bundles with even number of singular fibers, then
for any $\alpha\in H_2^-(X(0)\sm B)$ one has
$$
\hat q_{\theta^{X(t_1)}}(\mu_{t_1}(\alpha))=\hat q_{\theta^{X(t_2)}}(\mu_{t_2}(\alpha)).
$$
\end{proposition}

\begin{proof} By definition of coherence, there exists $\theta\in\Pic^-(\P(\E))$ such that $\theta^{X(t_1)}=\inh^{X(t_1)}(\theta)$ and $\theta^{X(t_2)}=\inh^{X(t_2)}(\theta)$.
On the other hand, their respective restrictions to $X_\R(t_i)\sm B$ coincide with introduced above $\theta^{t_i}$. Therefore 
both $\hat q_{\theta^{X(t_1)}}\circ\mu_{t_1}$
and $\hat q_{\theta^{X(t_2)}}\circ\mu_{t_2}$
are included in the continuous family of quadratic functions
$\hat q_{\theta^{t}}\circ \mu_t : H_2^-(X(t))\cap [E(t)]^\perp = H_2^-(X(0)\sm B) \to \Z/4$. But any continuous family of quadratic $\Z/4$-valued functions on a finitely generated $\Z/2$-vector spaces is constant.
\end{proof}

All of the above extends to real embedded conic bundles, $X\subset \P(\E)$, with an odd number of singular fibers, but requires a special adjustment caused by the need to fix one of real nonsingular fibers
(see Section \ref{inh-odd}).

So, in the case of odd $d$, we
consider similarly a pair of
real embedded conic bundles $X_1, X_2\subset \P(\E)$ belonging to the same linear system $2M+dF$, and related by wall-crossing.
 For simplicity  (without loss of generality), 
 we can assume that their fixed fibers $C_k= X_k\cap \P(\E_{\infty\R})$
 coincide and put $C=C_{1}=C_{2} $.

The connecting real Morse-Lefschetz family $\pi : \Cal X\to \mathbb D$ formed by  $X(z)\in \vert 2M+dF\vert $, $X_k=X(t_k)$ for some $t_k\in\DD_\R\sm \{0\}$, $k=1,2$,
is defined by a unique (up to a $\conj$-equivariant  factor $g: \DD\to \C\-\{0\}$) real equation $\sum A_{ij}(u,v,t) x_ix_j=0$ with coefficients depending analytically on $t$ and
polynomial homogeneous in $u,v$. Here, we choose  $\{ \sum A_{ij}(u,1,t) x_ix_j\ge 0\}$ as a half of $\P(\E_\R)\times D_\R$ having $\X_\R\sm  (C_\R\times D_\R)$ as a boundary,
and let $\inh^\X_C : \Pin^-(\P(\E_\R))=\Pin^-(\P(\E_\R)\times D_\R)\to \Pin^-(\X_\R\sm (C_\R\times D_\R))$ to be defined by this choice.

\begin{lemma}\label{surjectivity-ML}
The restriction mapping $\res : \Pin^-(\X_\R)\to \Pin^-(\X_\R\sm (C_\R\times D_\R))$ is surjective.
\end{lemma}
\begin{proof} Here, $\X_\R$ is a manifold of dimension 3, and therefore, as is well known, $\Pic^-(\X_\R)\ne \emptyset$.
Thus, like in the proof of Lemma \ref{surjectivity-CB}, the result follows from surjectivity of the restriction homomorphism
 $H^1(\X_\R)\to H^1(\X_\R\sm (C_\R\times D_\R))$ .
\end{proof}

We define $\Pin^-_{\inh}(\X_\R, C_\R, \E)\subset \Pin^-(\X_\R)$ to be 
$\res^{-1}(\inh^\X_C(\Pin^-(\P(\E_\R))))$.
We say that $\Pin^-$-structures $\theta^{X_k}\in \Pin^-_{\inh} (X_{k\R},C_\R,\E)$, $k=1,2$,
{\it are coherent}
(with respect to the chosen family $\Cal X$),
if there exists $\theta\in \Pin^-_{\inh}(\X_\R, C_\R, \E)$ such that the reduction maps $\Pin^-(\X_\R)\to \Pin^-(X_{k\R})$, $k=1,2$,
respecting the coorientation given by $\frac{\partial}{\partial t}$ send $\theta$ to $\theta^{X_k}$.
Clearly, such a coherence depends only on the choice of the connecting Morse-Lefschetz family.

\begin{proposition}\label{exist-coher-ML-odd}
For any real embedded conic bundle $X\subset \P(\E)$ belonging to a real Morse-Lefschetz family of conic bundles with odd number of singular fibers
and any $\theta^X\in \Pin^-_{\inh,{\rm mon}}(X_\R,C_\R,\E)$, 
there exists a family of coherent inherited $\Pin^-$-structures on this Morse-Lefschetz family that coincides with $\theta^X$ on $X$.
\end{proposition}
\begin{proof} 
As it follows from Lemma \ref{surjectivity-ML}, the set  $\Pin^-_{inh}(\X_\R, C_\R, \E)$ is an affine space over a subgroup of $H^1(\X_\R; \Z/2)$ spanned by $w_1(\X_\R)$ and
$[C_\R\times D_\R]^*$. On the other hand, by Lemma \ref{torsor-odd}, $\Pin_\inh^-(X_\R,C_\R,\E)$ is an affine space with respect to the subgroup of
$H^1(X_\R;\Z/2)$  spanned by $w_1(X_\R)$ and  $[C_\R]^*$. So, it remains to notice that the restriction
homomorphism $H^1(\X_\R;\Z/2)\to H^1(X_\R;\Z/2)$ sends $[C_\R\times D_\R]^*$ to  $[C_\R]^*$, and (because of triviality of the normal bundle of $X_\R$ in $\X_\R$)
$w_1(\X_\R)$ to $w_1(X_\R)$.
%
\end{proof}

\begin{lemma}\label{equality-coher-ML-odd}
If 
$\theta^{X(t_k)}\in \Pin^-_{\inh,{\rm mon}}(X_\R(t_k),C_\R,\E)$
with $t_k\in\DD_\R\sm\{0\},\ k=1,2$ 
 are coherent with respect to a real Morse-Lefschetz family $\pi : \Cal X\to \mathbb D$
of conic bundles with odd number of singular fibers, then
for any $\alpha\in H_2^-(X(0)\sm B)$ we have
$$
\hat q_{\theta^{X(t_1)}}(\mu_{t_1}(\alpha))=\hat q_{\theta^{X(t_2)}}(\mu_{t_2}(\alpha)).
$$
\end{lemma}
\begin{proof} It coincides literally with that of Proposition \ref{equality-coher-ML-even}.
\end{proof}

\begin{remark} If one of two coherent (with respect to a real Morse-Lefschetz family) inherited structures is monic, then the other is monic too. To justify this, it is sufficient
to notice that $w_1^*(X_\R(t))$ is supported in the complement of a vanishing cycle.
\end{remark}

\subsection{Push-forward formula}\label{push}
Assume that real  embedded conic bundles $X, Y\subset  \P(\E)$ belonging to the same linear system $2M+dF$, 
$d\in\Z_{\ge 0}$,
are related by wall-crossing,
and consider a connecting them real Morse-Lefschetz family $\pi : \Cal X\to \mathbb D$ formed by  $X(z)\in \vert 2M+dF\vert $, $X=X(t_1)$, $Y=X(t_2)$ for some $t_1, t_2\in\DD_\R$, $t_1<0<t_2$.
To shorten further notation, let us identify the
lattices $H_2(X;\Z)$ and $H_2(Y;\Z)$ using a smooth trivialization of the fibration $\X\to\DD$
over the punctured upper half disc $\DD^+=\{z\in \DD, \Im z\ge 0, z\ne 0\}$. Clearly, this identification preserves the canonical class $K$, the class of the fibers $C$, and the vanishing class $e$.
The involutions $c_t:
H_2(X(t))\to H_2(X(t))$ 
for $t=t_1,t_2$ with $t_1<0<t_2$
become related by the Picard-Lefschetz transformation:
\begin{equation}
c_{t}=s_e\circ c_{-t} \quad s_e(v)=v+(ev)e,
\end{equation}
where $e\in H_2(X(t))$ is the vanishing class of the underlying nodal degeneration.

Suppose that the direction of wall-crossing is chosen so that $\chi(X_\R)<\chi(Y_\R)$, which means that
the Morse index of the modification $X_\R\to Y_\R$ has either index $2$ (pinching of a circle), or index $0$ (birth of a spherical component).
Consider $\conj$-invariant vanishing cycles (2-spheres)
$E^X\subset X$ and $E^Y\subset Y$ and their traces on the real loci, $E^X_\R=E^X\cap X_\R$ and $E^Y_\R=E^Y\cap Y_\R$.
Under above identification of $H_2(Y;\Z)$ with $H_2(X;\Z)$ and the chosen direction of wall-crossing, we have
$$
H_2^-(Y)=\{\alpha\in H_2^-(X)\,|\, \alpha e=0\} 
$$
We consider the orthogonal projection
\begin{equation}\label{projection}
p^e: H_2^-(X)\to H_2^-(Y)\otimes\Q,\quad \text{given by}
\qquad v\mapsto v+\frac12(ev)e,
\end{equation}
and the functions
$\W_k^{X,\theta^X}: H_2^-(X)\to \Z$, $\W_k^{Y,\theta^Y}: H_2^-(Y)\to \Z$ with $k\in \Z$ defined by
$\W_k^{X,\theta^X}(\a)=i^{\hat q_X(\a)-m^2}w_k({\a}), \a\in  H_2^-(X)$ and, respectively,
$\W_k^{Y,\theta^Y}(\a)=i^{\hat q_Y(\a)-m^2}w_k({\a}), \a\in H_2^-(Y)$ (see Section \ref{welsch-prelim} for notation $w_k({\a})$).

The push-forward operation treated in the next proposition transforms functions 
 $h:H_2^-(X)\to \Z$ into functions $p^e_*(h): H_2^-(Y)\otimes\Q\to\Z$ defined by 
$p^e_*(h)(y)=\sum_{x\in (p^e)^{-1}(y)}h(x)$. It is properly defined, if $h$
 has finite support in $(p^e)^{-1}(y)$, for each  $y\in H_2^-(Y)\otimes\Q$.

The proof of this proposition is literally the same as the proof of \cite[Prop. 3.1.1]{surgery-invariant}, with a replacement (at the very end of the proof)
of the reference to \cite[Prop. 2.3.2]{surgery-invariant} by reference to Proposition \ref{equality-coher-ML-even} if $d$ is even and 
\ref{equality-coher-ML-odd} if $d$ is odd.
Therefore we omit the proof.

 \begin{proposition}\label{push-forward}
 Assume that
 real conic bundles $X, Y\subset \P(\E)$ with $s$ singular fibers
 are related by a Morse-Lefschetz family, 
 both $X_\R$, $Y_\R$ are non-empty and 
 $\chi(X_\R)<\chi(Y_\R)$.
Let  $\theta^X\in\Pin^-(X_\R)$, $\theta^Y\in\Pin^-(Y_\R)$ be coherent $\Pin^-$-structures taken from 
$\Pin^-_{\inh}$ if $s$ is even and $\Pin^-_{\inh, {\rm mon}}$ if odd.
 When $s$ is odd,
 assume also that the fixed conic-fibers in the family
 are placed away from the node that appears in this family.
Then:
\begin{enumerate}
\item The push-forward $p^e_* (\Cal W_k^{X,\theta^X})$ is properly defined.
\item
It is supported on
$H_2^-(Y)=H_2^-(Y)\otimes1\subset H_2^-(Y)\otimes\Q$.
\item
$p^e_* (\Cal W_k^{X,\theta^X})=\Cal W_k^{Y,\theta^Y}$ on $H_2^-(Y)$. \qed
\end{enumerate}
\end{proposition}

\begin{cor}\label{wall-crossing-N}
For $X$, $Y$ and $\theta^X$, $\theta^Y$ as in the assumptions of Proposition \ref{push-forward},
the numbers $\NN_{m,n}(X,\theta^X)$ and $\NN_{m,n}(Y,\theta^Y)$ coincide for all $m,n\in\Z$.
\end{cor}

\begin{proof}
It follows from Proposition \ref{push-forward},
since $(m,n)$ is preserved by $p^e_*$ and 
\[\NN_{m,n}(X,\theta^X)=\hspace{-2mm}\sum_{\a\in\LL^{m,n}_\R(X)}\hspace{-2mm}\W_k^{X,\theta^X}(\a),\quad k\in\{0,1\},\quad k=m-1\,\rm{mod}\,2.\qedhere\]
\end{proof}

\subsection{Invariance of $\NN_{m,n}$
and proof of Theorems \ref{main-main}, \ref{refinement}}

\begin{theorem}\label{main-new}
For  any fixed $s\ge 2$, the double sequence $\NN_{m,n}(X,\theta^X)$ does not depend
on a choice of a real conic bundle $X\subset\P(\E^s)$, $X_\R\ne\oo$,
from the linear system $\XX^s$ and a choice of a $\Pin^-$-structure $\theta^X$ taken in
 $\Pin^-_{\inh}(X,\E^s)$ if $s$ is even and in $\Pin^-_{\inh,{\rm mon}}(X,\E^s,C)$ (with arbitrary choice of a fiber $C$) if $s$ is odd.
Moreover, these numbers vanish if $n$ is odd, and in case of even $s$ they vanish also if $m$ is odd.
\end{theorem}

\begin{proof}
As it follows from Theorem \ref{exhaustive},  there exists a real conic bundle $Y$ which belongs to the same linear system $\XX^s$, has $Y_\R=S^2$
if $s$ is even, and has $Y_\R=\Rp2$ otherwise. On the other hand, Proposition \ref{connecting} and Corollary \ref{wall-crossing-N} (to make its assumptions satisfied we use Lemma \ref{exist-coher-ML-even} if $s$ is even and Lemma \ref{exist-coher-ML-odd} if odd)
imply $\NN_{m,n}(X,\theta^X)$ $=\NN_{m,n}(Y,\theta^Y)$.
Thus, due to uniqueness of $\Pin^-$-structures on $S^2$ and that for $\Pin^-$-structures on $\Rp2$ with $q_\theta(w_1^*)=1$,
we conclude that $\NN_{m,n}(X,\theta^X)$ does not depend on $\theta^X$ and on $X$ (for a fixed $s$).
Furthermore, in both cases some of the real fibers of $Y$ have no real points, and therefore, for any real divisor class $D\in H_2(Y)$ we get $D C\,\overset{{\rm mod}\,2}=\,D_\R C_\R=0$,
which implies that $\NN_{m,n}=0$ for odd $n$. Similarly, if $s$ is even,
$D K\,\overset{{\rm mod}\,2}=D_\R K_\R=0$ because $H_1(Y_\R;\Z/2)=0$, wherefrom we get vanishing of $\NN_{m,n}=0$ for odd $m$.\,
 \end{proof}

\begin{proof}[Proof of Theorem \ref{main-main}]
The first part of Theorem \ref{main-main} and
vanishing of $\NN_{m,n}(X,\theta)$ under the indicated parity conditions
follow from Theorem \ref{main-new} and Corollary \ref{wall-invariance}.
The vanishing for $n<0$  follows from non-negativity of the intersection 
of effective classes with the fibers of our conic bundle.
For the vanishing in the case $m\le 0$, see \cite[Lemma 4.2]{GP}.
\end{proof}

\begin{proof}[Proof of Theorem \ref{refinement}] Recall that, by definition,  for any real conic bundle $Y$ we have 
$ \lay^{m}_\R(Y)=\cup_{n\ge 0}\lay^{m,n}_\R(Y)$.  
Thus, as in the proof of Theorem \ref{main-new}, it is sufficient to choose $Y$ with 
$Y_\R=S^2$ (if $s$ is even) or $\Rp2$ (if odd)
and to use uniqueness of $\theta$ (monic if $s$ is odd).
 \end{proof}

\subsection{An invariance of the right-hand side of (\ref{mag-formula})}\label{right-side}
By $\DD_k, k\ge 0$, we denote the  lattice
$\{\bar x\in\Z^k\,|\,\sum_ix_i\in2\Z\}$, where $\Z^k$ is equipped with the quadratic form $-\sum_i x_i^2$.
It gives $\DD_0=0$, $\DD_1=\la 4\ra$,
$\DD_2=\AA_1\oplus \AA_1$ and $\DD_3=\AA_3$.

\begin{lemma}\label{root-vanishing} For any conic bundle with $s \ge 2$ singular fibers the lattice $(K,C)^\perp$ is isomorphic to $\DD_s$, and 
each of its roots represents a vanishing cycle.
\end{lemma}

\begin{proof}
As is known (and easy to check by taking as representative examples generic blowups $X=\P^2\#(s+1)\bar\P^2$ with $s\ge  2$
and  the conic bundle $H-E$), 
$(K,C)^\perp$ admits a geometric basis consisting of roots represented by vanishing cycles. 
Therefore, the action of the monodromy group  on $H_2(X)$ contains the Weyl group of $(K,C)^\perp=\DD_s$. 
The latter group is generated by reflections in the basic roots, and in the case of $\DD_s$ with $s\ge 3$ it acts transitively on the set of roots, see  \cite[Ch. 6, §1.3, Prop. 11]{Bour}.
\end{proof}

\begin{proposition}\label{root-moments} For any embedded conic bundle $X$ with 
$s\ge 3$ singular fibers,
the sum $\sum_{\a\in\Cal L^{m,n}}(e\cdot \a)^2 GW(\a)$ does not depend on the choice of a root $e\in (K,C)^\perp$.
Furthermore for every
$u,v\in (K,C)^\perp$, we have
$$
u^2\sum_{\a\in\Cal L^{m,n}}(v\cdot \a)^2 GW(\a)= v^2\sum_{\a\in\Cal L^{m,n}}(u\cdot \a)^2 GW(\a).
$$
\end{proposition}

\begin{proof} 

The independence of $\sum_{\a\in\Cal L^{m,n}}(e\cdot \a)^2 GW(\a)$ from the choice of $e$  follows from invariance of $GW(\a)$ 
and $\Cal L^{m,n}$ under deformation of $X$ as a conic bundle combined with transitivity of the action of the Weyl group of $\DD_s$,
$s\ge 3$,
on the set of roots.

The second statement follows from the same arguments plus the fact that for $s\ge 3$
the function $\DD_s \to \Z$, $x\mapsto x^2$ is the only 
(up to scalar factor) quadratic function invariant under the action of the Weyl group
(as it follows from the irreducibility of the action, see \cite[§2.1, Prop. 1]{Bour}).
\end{proof}
 
 \subsection{Proof of Theorem \ref{magic-theorem}}\label{proof-magic}
 This proof repeats nearly verbatim the arguments from the proof of Theorem 1.2.3 in \cite{surgery-invariant}. 
 
 First of all, we notice that due to Theorem \ref{main-main} and Proposition \ref{root-moments}
 it is sufficient to check the formula (\ref{mag-formula})
for some particular example of a
real conic bundle $X$ with
a given number $s\ge4$ of singular fibers
and $X_\R\ne\oo$, with
 some 
particular choice of a lattice root in $(K,C)^\perp\subset H_2(X)$.

An appropriate set of examples
is given by Proposition \ref{spherical}. 
It provides, for each $s\ge4$,
a real Morse-Lefschetz family
$f : \X\to \mathbb D$, $\DD=\{z \in\C\,|\,| z |<1\}$ of embedded conic bundles $X(z)\subset \P(\Cal E)$ such that for real $0< z <1$
a spherical component $S^2\subset X_\R(z)$
is collapsing to the unique node of $X(0)$ as $z\to 0^+, z\in\R$. 
In accord with 
the construction given in Proposition \ref{spherical}, we choose a continuous $\conj$-invariant family of embedded nonsingular rational curves $C(z)\subset X(z)$, $z\in\mathbb D$,  such that, for each 
 $z\in\mathbb D$,
 $C(z)$ is a fiber 
of  $\pi(z): X(z)\to \P^1$, and $C(0)$ is disjoint from the node of $X(0)$.
Then, we follow  the proof from \cite{surgery-invariant} passing from algebraic to symplectic setting, use the curves $C(z)$ for defining, by duality,
a dimension 2 cohomology class on the total space of the symplectic fibration described there. This class extends the notion of conic level uniformly to
all the fibers of the latter fibration including the singular fiber, with the property to vanish on every homology class of the $\P^1\times\P^1$-component of the singular fiber.
After that, it is sufficient to repeat 
the proof from \cite{surgery-invariant} restricting the consideration of involved real (respectively, complex)
curves to those of conic-level $2n$ (respectively, $n$).
\qed

\subsection{Positivity}\label{positivity}
Recall that due to Theorem \ref{main-main} the bivariate sequence $\NN_{m,n}$
is well defined for any real rational surface with a fixed $K^2\le 6$.

\begin{theorem} $\NN_{m,n}\ge 0$ for any $m,n\in\Z$.
\end{theorem}

\begin{proof} Assume, first, that $K^2$ is $\le 6$ but $\ne 5$ (which in terms of conic bundles means that the number of singular fibers is $\ge 2$ and $\ne 3$).
Take as an example a real conic bundle $X$
with $X_\R=\Rp2\+S^2$ or $S^2$. In such a case,
all curves are counted with sign $+1$, which is
due to $i^{\hat q(\alpha)-m^2}=1$ and $w_k(\alpha)\ge0$ for any $\alpha\in \lay^{m,n}_\R(X)$ (the former relation follows from the choice of monic $\Pin^-$-structures when $X_\R=\Rp2\+S^2$,
while the latter one is due  to Welschinger positivity theorem \cite[Theorem 1.1]{optimality}).
In the case $K^2=5$ there exists an example of a real conic bundle with $X_\R=\Rp2$, to which
the Welschinger theorem is also applicable (for that value of $K^2$), and the same argument works.
\end{proof}

\begin{remark} For $K^2\le 4$, the positivity of $\NN_{2m,2n}$ for all $m,n\in\Z$
is also a straightforward consequence of Theorem \ref{magic-theorem} 
and positivity of $GW^{[2]}_{m,n}$.
\end{remark}

\section{Recursion formulas}\label{sect-recursion}

\subsection{Values at the levels $(m,0)$}

\begin{proposition}\label{GWm0}
For any real conic bundle $X$ with $s\ge2$ singular fibers
the numbers $GW_{m,0}, GW_{m,0}^{[2]}$ and
$ \NN_{m,0}$, $m\ge1$  vanish, except 
$$
GW_{1,0}=2s,\quad GW_{2,0}=1,\quad GW_{1,0}^{[2]}=4,\quad \NN_{2,0}=1,
$$
and,  in the case of odd $s$,  $\NN_{1,0}=2$.
\end{proposition}

\begin{proof}
For any effective divisor $D\subset X$, vanishing of $n=D\cdot C$ implies that $D$ is a component of a conic fiber.
This gives the values of $GW_{m,0}$, for $m=1,2$ along with $GW_{m,0}=GW_{m,0}^{[2]}=0$ for $m>2$.
To check $GW_{1,0}^{[2]}=4$, $GW_{2,0}^{[2]}=0$, one can take as an example a real conic bundle
with a spherical component and observe that the spherical component represents a vanishing cycle.
This cycle does not meet some conic-fibers, which gives  $GW_{2,0}^{[2]}=0$,
and meets
transversely the 4 irreducible components of two singular fibers and is disjoint from other singular fibers, which implies $GW_{1,0}^{[2]}=4$.

By Theorem \ref{magic-theorem}, $\NN_{2m,0}=2^{m-3}GW^{[2]}_{m,0}$, while $\NN_{2m+1,0}=0$ if $s$ even  by Theorem \ref{main-main}.
For the remaining formula $\NN_{1,0}=2$ in the case of odd $s\ge2$ one can take as an example a real conic bundle with one and only one real singular fiber. Then,
the only real effective divisors of bi-level $(1,0)$ are the irreducible components of this real singular fiber. Each of these components contribute $1$ to $\NN_{1,0}$, since the $\Pin^-$-structure on 
$X_\R$ is monic and vanishes on the conic-fibers.
\end{proof}

We obtain these relations using the following statement, which is a direct consequence of the associative law for the small quantum product.

\subsection{Recursion for $GW_{m,n}$.}\label{recursion-any-d}
The following statement  is a direct consequence of the associative law for the small quantum product.

\begin{proposition}\label{prop-associativity}
Let $X$ be a rational surface. Then, for any $H_1, H_2, H_3, H_4\in H_2(X)$ and any $\alpha \in H_2(X)$ with $-K\alpha\ge 2$, the following relation holds:
\begin{equation}\label{associativity}
\begin{aligned}
(H_1H_2)(\alpha H_3)(\alpha H_4) GW(\alpha) + (H_3H_4)(\alpha H_1)(\alpha H_2)GW(\alpha) + &\\
\sum_{\alpha=\alpha_1+\alpha_2} GW(\alpha_1)GW(\alpha_2)\binom{m-2}{m_1-1}(\alpha_1H_1)(\alpha_1H_2)(\alpha_2H_3)(\alpha_2H_4)(\alpha_1\alpha_2)= & \\
(H_1H_3)(\alpha H_2)(\alpha H_4) GW(\alpha) + (H_2H_4)(\alpha H_1)(\alpha H_3)GW(\alpha) + &\\
\sum_{\alpha=\alpha_1+\alpha_2}  GW(\alpha_1)GW(\alpha_2)\binom{m-2}{m_1-1}(\alpha_1H_1)(\alpha_1H_3)(\alpha_2H_2)(\alpha_2H_4)(\alpha_1\alpha_2)
\end{aligned}
\end{equation}
where $m=-K\alpha$ and $m_1=-K\alpha_1$.\qed
\end{proposition}

\begin{lemma}\label{zero-sum}
For $K^2\le 6$ and any $m\in \Z_{\ge 1}, n\in\Z_{\ge 0}$, we have
\begin{equation}
\sum_{\substack{v\in (K,C)^\perp\otimes\Q\\ \beta\in (K,C)\otimes \Q, -K\beta=m, C\beta=n
\\ \beta+v\in H_2(X;\Z)\\}}
\hskip-8mm
GW(\beta+v)v =0.
\end{equation}
\end{lemma}
\begin{proof} 
For any root vector $e\in (C,K)^\perp$, 
the invariance of Gromov-Witten invariants under Picard-Lefschetz transformations, implies
\begin{equation*}
\sum_{\substack{v\in (K,C)^\perp\otimes\Q\\ \beta\in (K,C)\otimes \Q, -K\beta=m, C\beta=n
\\ \beta+v\in H_2(X;\Z)\\}}
\hskip-8mm
GW(\beta+v)v =\sum_{\substack{v\in (K,C)^\perp\otimes\Q\\ \beta\in (K,C)\otimes \Q, -K\beta=m, C\beta=n
\\ \beta+v\in H_2(X;\Z)\\}}
\hskip-8mm
GW(\beta+v)(v+(ev)e),
\end{equation*}
which in its turn implies
\begin{equation*}
\sum_{\substack{v\in (K,C)^\perp\otimes\Q\\ \beta\in (K,C)\otimes \Q, -K\beta=m, C\beta=n
\\ \beta+v\in H_2(X;\Z)\\}}
\hskip-7mm
GW(\beta+v)(ev)e=0.
\end{equation*}
Thus, for any root vector  $e\in (K,C)^\perp$ we have
\begin{equation*}
e\cdot\hskip-15mm\sum_{\substack{v\in (K,C)^\perp\otimes\Q\\ \beta\in (K,C)\otimes \Q, -K\beta=m, C\beta=n
\\ \beta+v\in H_2(X;\Z)\\}}
\hskip-7mm
GW(\beta+v)v=0
\end{equation*}
Since $(K,C)^\perp$ is a non-degenerate lattice and it is generated by the root vectors (see Lemma \ref{root-vanishing}),
the 
result stated follows from the latter relation.
\end{proof}

{\sl Convention.} In what follows, to simplify writing sums with many rules or limits on parameters, we authorize, when working under summation with expressions like $GW(\beta+v)$ to take vectors $\beta\in(K,C)\otimes \Q$, $v\in (K,C)^\perp\otimes\Q$ with a convention that $GW(\beta+v)=0$ if $\beta, v$ are not respecting the gluing rule $\beta+v \in H_2(X)$.

\begin{theorem}\label{simple-recursion-any-d}  For 
any $d=K^2\le 6$
and any 
$m\in \Z_{\ge 2}, n\in\Z_{\ge 1}$
the following 
 holds
\begin{equation}
\begin{aligned}
&(4mn-dn^2)\,GW_{m,n}=
\sum_{\substack{m=m_1+m_2 \\ n=n_1+n_2}}GW_{m_1,n_1} GW_{m_2,n_2}\binom{m-2}{m_1-1} \cdot \\
&(m_1-\frac{n_1d}2)n_2
(m_1n_2-m_2n_1)(\frac12(m_1n_2+n_1m_2)-\frac14 n_1n_2d).
\end{aligned}
\end{equation}

If $m\ge 3$ and $mn-\frac{dn^2}4\le 0$, or $m=2$ and $mn-\frac{dn^2}4 < 0$, then $GW_{m,n}=0$. For $m=2$ and $mn-\frac{dn^2}4 = 0$, the only non zero values of $GW_{m,n}$
are $GW(2, \frac8{d})=1$ with $d=1, 2$, and $4$.
\end{theorem}

\begin{proof}
We put $K'=\frac{d}2 C+K$ and  choose $H_1=-K'$, $H_2= C-e$, $H_3=-K'-e$, $H_4= C$ where $e\in (K,C)^\perp$ is a root vector.
 Then we apply Proposition \ref{associativity}
 to  $\alpha\in H_2(X)$ written in a form $\alpha=\frac{m}2C-\frac{n}2K'+v$ with $v\in (K,C)^\perp\otimes \Q$.
This gives
\begin{equation}
\begin{aligned}
&(2(m-\frac{nd}2 -ev)n+2(m-\frac{nd}2)(n-ev)+dn(n-ev))GW(\alpha)=\\
&\sum_{\substack{m=m_1+m_2\\ n=n_1+n_2\\ v=v_1+v_2\\v_1,v_2\in (K,C)^\perp\otimes\Q}}
\hskip-7mm
GW(\alpha_1)GW(\alpha_2)\binom{m-2}{m_1-1}(m_1-\frac{n_1d}2 )n_2\big((m_1-\frac{n_1d}2 -ev_1)(n_2-ev_2)- \\
&(n_1-ev_1)(m_2-\frac{n_2d}2-ev_2)\big) (\frac12(m_1n_2+m_2n_1)-\frac{n_1n_2d}4+v_1v_2)
\end{aligned}
\end{equation}
where $\alpha_j$, $j=1,2$, stands for $\frac{m_j}2C-\frac{n_j}2K'+v_j$.  Taking the average with respect to $\pm e$ we get
\begin{equation}
\begin{aligned}
&(4(m-\frac{nd}2)n+dn^2)\,GW(\alpha)=\\
&\sum_{\substack{m=m_1+m_2\\ n=n_1+n_2\\ v=v_1+v_2\\v_1,v_2\in (K,C)^\perp\otimes \Q}}
\hskip-4mm
GW(\alpha_1)GW(\alpha_2)\binom{m-2}{m_1-1}(m_1-\frac{n_1d}2 )n_2\big((m_1-\frac{n_1d}2 )n_2-\\
&n_1(m_2-\frac{n_2d}2)\big)(\frac12(m_1n_2+m_2n_1)-\frac{n_1n_2d}4+v_1v_2).
\end{aligned}
\end{equation}
Then summing over all $v\in (K,C)^\perp\otimes\Q$ we obtain
$$
\begin{aligned}
&(4mn-dn^2)GW_{m,n}\!=
\hskip-6mm
\sum_{\substack{m=m_1+m_2\\ n=n_1+n_2}}
\hskip-2mm
\binom{m-2}{m_1-1}(m_1-\frac{n_1d}2 )n_2(m_1n_2-n_1m_2)
\big(\big(\frac{m_1n_2+m_2n_1}2 \\
&-\frac{n_1n_2d}4)
\hskip-5mm
\sum_{
v_1,v_2\in (K,C)^\perp\otimes\Q}
\hskip-5mm
GW(\alpha_1)
 GW(\alpha_2)\big) +
 \hskip-7mm
\sum_{
v_1,v_2\in (K,C)^\perp\otimes\Q}
GW(\alpha_1)GW(\alpha_2)(v_1v_2)\big).
\end{aligned}
$$
Finally, due to Lemma \ref{zero-sum},
\begin{equation*}
\sum_{
v_1,v_2\in (K,C)^\perp\otimes\Q}
\hskip-5mm
GW(\alpha_1)GW(\alpha_2)(v_1v_2)= 
\hskip-5mm
\sum_{
v_1\in (C,K)^\perp\otimes\Q}
\hskip-5mm
GW(\alpha_1)v_1\,\,\cdot 
\hskip-5mm
\sum_{
v_2\in (C,K)^\perp\otimes\Q}
\hskip-5mm
GW(\alpha_2)v_2= 0
\end{equation*}
and the required recursion relation follows.

To prove the stated vanishing of $GW(\alpha)$ with $\alpha=\frac{m}2C-\frac{n}2K'+v$ we note that $\alpha^2=mn-\frac{dn^2}4+v^2\le mn-\frac{dn^2}4$
and argue as follows. If $m\ge 3$ then the counted rational curves $D$ with $[D]=\alpha$ are constrained by at least 2 points, so that moving one these points
and keeping the others implies $\alpha^2=D^2 >0$, contradicting the assumption $mn-\frac{dn^2}4\le 0$. Similarly, if $m\ge 2$, moving one of the point-constraints
we conclude that $mn-\frac{dn^2}4\ge 0$, contradicting assumption  $mn-\frac{dn^2}4<0$. Finally, if $m=2$ and $GW(\alpha)\ne 0$, then, by above, $\alpha^2\ge 0$,
which under assumption $mn-\frac{dn^2}4=0$ implies $v=0$ and $dn=8$. Thus, there remains to consider $\alpha= -C-\frac{n}2K$ with $d=4, 2$, and $1$,
and to observe that $GW(2, \frac{8}d)=1$ in each of these 3 cases (cf. Remark \ref{palyndrom}).
\end{proof}

\subsection{Recursion for $GW^{[2]}_{m,n}$.}
\label{moment-any-d}
\begin{lemma}\label{vector-moment} Let $K^2\le 5$. Then, for
any $e\in (K,C)^\perp$ with $e^2=-2$ we have
\begin{equation}
\sum_{v\in (K,C)^\perp\otimes\Q} GW(\frac{m}2 C-\frac{n}2K'+v)(ev)v=-\frac12 GW_{m,n}^{[2]}e.
\end{equation}
\end{lemma}
\begin{proof} Due to invariance of $GW(\frac{m}2 C-\frac{n}2K'+v)$ with respect $v\mapsto v+e(ve)$
we have
\begin{equation}
\begin{aligned}
&\sum_{v\in (K,C)^\perp\otimes\Q} GW(\frac{m}2 C-\frac{n}2K'+v)(ev)v=\\
&\sum_{v\in (K,C)^\perp\otimes\Q} GW(\frac{m}2 C-\frac{n}2K'+v)(e(v+e(ve))) 
(v+e(ve))=\\
&-\hskip-5mm
\sum_{v\in (K,C)^\perp\otimes\Q} GW(\frac{m}2 C-\frac{n}2K'+v)(ev)v - 
\hskip-5mm
\sum_{v\in (K,C)^\perp\otimes\Q} GW(\frac{m}2 C-\frac{n}2K'+v)(ev)^2 e
\end{aligned}
\end{equation}
which gives the result stated.
\end{proof}

\begin{theorem}\label{moment-recursion-any-d}
Let $K^2\le 5$.
Then,
for any $m\ge 2, n\ge 0$ the following recursion relation holds
\begin{equation}
\begin{aligned}
&(2(m-\frac{nd}2)^2+(2+d)n^2)GW_{m,n}+(4+d)GW^{[2]}_{m,n}=\\
&\sum_{\substack{m=m_1+m_2\\ n=n_1+n_2}}\big(a\,GW_{m_1,n_1}GW_{m_2,n_2}+b_1GW^{[2]}_{m_1,n_1}GW_{m_2,n_2}+\\
&b_2 GW_{m_1,n_1}GW^{[2]}_{m_2,n_2} +c \, GW^{[2]}_{m_1,n_1}GW^{[2]}_{m_2,n_2}\big) \\
\end{aligned}
\end{equation}
where
$$
\begin{aligned}
 &a= \binom{m-2}{m_1-1} (\frac{m_1n_2+m_2n_1}2- \frac{n_1n_2d}4)((m_1-\frac{n_1d}2)^2n_2^2- (m_1-\frac{n_1d}2)(m_2-\frac{n_2d}2)n_1n_2 ) \\
& b_1= \binom{m-2}{m_1-1}(\frac{m_1n_2+m_2n_1}2- \frac{n_1n_2d}4)(n_2^2-( m_2-\frac{n_2d}2)n_2 ) \\
&b_2=\binom{m-2}{m_1-1} (\frac{m_1n_2+m_2n_1}2- \frac{n_1n_2d}4)((m_1-\frac{n_1d}2)^2- (m_1-\frac{n_1d}2)n_1 ) \\
&c= \frac12 \binom{m-2}{m_1-1} ((m_1-\frac{n_1d}2)(m_2-\frac{n_2d}2)+n_1n_2+n_1(m_2-\frac{n_2d}2)-3(m_1-\frac{n_1d}2)n_2 ).  \\
\end{aligned}
$$
\end{theorem}
\begin{proof}
We put $K'=\frac{d}2 C+K$ and  choose $H_1=H_3=-K'-e$, $H_2= H_4=C-e$,  where $e\in (K,C)^\perp$, $e^2=-2$, and apply Proposition \ref{associativity} to  $\alpha\in H_2(X)$ written in a form $\alpha=\frac{m}2C-\frac{n}2K'+v$ with $v\in (K,C)^\perp\otimes\Q$.
This gives
\[
\begin{aligned}
&(2(m-\frac{nd}2-ev)^2+(2+d)(n-ev)^2)GW(\alpha)=\\
&\sum_{\substack{(m,n)=(m_1,n_1)+(m_2,n_2)\\ v=v_1+v_2,v_1,v_2\in (K,C)^\perp\otimes\Q}}
\hskip-8mm
GW(\alpha_1)GW(\alpha_2)\binom{m-2}{m_1-1}((m_1-\frac{n_1d}2-ev_1)^2(n_2-ev_2)^2- \\
&
(m_1-\frac{n_1d}2-ev_1)(n_1-ev_1)(m_2-\frac{n_2d}2-ev_2)(n_2-ev_2)) (\frac{m_1n_2+m_2n_1}2-\frac{n_1n_2d}4+(v_1v_2))
\end{aligned}
\]
where $\alpha_j$, $j=1,2$, stands for $\frac{m_j}2C-\frac{n_j}2K'+v_j$. Taking the average with respect to $\pm e$ we get
%
\[
\begin{aligned}
&(2(m-\frac{nd}2)^2+(2+d)n^2+(4+d)(ev)^2)\,
GW(\alpha)=\\
&\sum_{\substack{(m,n)=(m_1,n_1)+(m_2,n_2)\\ v=v_1+v_2,v_1,v_2\in (K,C)^\perp\otimes\Q}}
\hskip-8mm
GW(\alpha_1)GW(\alpha_2)\binom{m-2}{m_1-1} (\frac12(m_1n_2+m_2n_1)- \frac{n_1n_2d}4+v_1v_2) \\
&((m_1-\frac{n_1d}2)^2n_2^2- (m_1-\frac{n_1d}2)(m_2-\frac{n_2d}2)n_1n_2 + (ev_1)^2(n_2^2-(m_2-\frac{n_2d}2)n_2)+\\
&(ev_2)^2((m_1-\frac{n_1d}2)^2-(m_1-\frac{n_1d}2)n_1)+(ev_1)(ev_2)\big(4(m_1-\frac{n_1d}2)n_2-(m_1-\frac{n_1d}2)n_2\\
&-(m_1-\frac{n_1d}2)(m_2-\frac{n_2d}2)-n_1n_2-n_1(m_2-\frac{n_2d}2 )\big).
\end{aligned}
\]
Then we sum over all $v_1,v_2\in (K,C)^\perp\otimes\Q$ and cancel the terms with $(ev_1)(ev_2), (v_1v_2)$, $(v_1v_2)(ev_1)^2, (v_1v_2)(ev_2)^2$ (that is justified by
splitting off a factor $\sum GW(\alpha_j)(ev_j)$, or $\sum GW(\alpha_j)v_j$ and applying Lemma \ref{zero-sum})
which gives
\begin{equation}
\begin{aligned}
&(2(m-\frac{nd}2)^2+(2+d)n^2)GW_{m,n}+(4+d)GW^{[2]}_{m,n}=\\
&\sum_{\substack{(m,n)=(m_1n_1)+(m_2,n_2)}}GW_{m_1,n_1}GW_{m_2,n_2}\binom{m-2}{m_1-1} (\frac12(m_1n_2+m_2n_1)- \frac{n_1n_2d}4)((m_1-\frac{n_1d}2)^2n_2^2-\\
&- (m_1-\frac{n_1d}2)(m_2-\frac{n_2d}2)n_1n_2 ) +
\sum_{\substack{(m,n)=(m_1,n_1)+(m_2,n_2)}}GW^{[2]}_{m_1,n_1}GW_{m_2,n_2}\binom{m-2}{m_1-1} \\
&(\frac12(m_1n_2+m_2n_1)- \frac{n_1n_2d}4)(n_2^2-( m_2-\frac{n_2d}2)n_2 ) +
\sum_{\substack{m=m_1+m_2\\ n=n_1+n_2}}GW_{m_1,n_1}GW^{[2]}_{m_2,n_2}\binom{m-2}{m_1-1}\\
& (\frac12(m_1n_2+m_2n_1)- \frac{n_1n_2d}4)((m_1-\frac{n_1d}2)^2- (m_1-\frac{n_1d}2)n_1 ) + \\
& \sum_{\substack{(m,n)=(m_1,n_1)+(m_2,n_2)\\ v=v_1+v_2,v_1,v_2\in (K,C)^\perp\otimes\Q}}GW(\alpha_1)GW(\alpha_2)\binom{m-2}{m_1-1}(v_1v_2)(ev_1)(ev_2) \\
& (3(m_1-\frac{n_1d}2)n_2-(m_1-\frac{n_1d}2)(m_2-\frac{n_2d}2)-n_1n_2-n_1(m_2-\frac{n_2d}2) ). \\
\end{aligned}
\end{equation}

Finally, it remains to apply Lemma \ref{vector-moment} allowing to replace the last sum by
\begin{equation*}
\begin{aligned}
&\frac12 \sum_{\substack{m=m_1+m_2\\ n=n_1+n_2}}GW^{[2]}_{m_1,n_1}GW^{[2]}_{m_2,n_2}\binom{m-2}{m_1-1} ((m_1-\frac{n_1d}2)(m_2-\frac{n_2d}2)+n_1n_2+\\
&n_1(m_2-\frac{n_2d}2)-3(m_1-\frac{n_1d}2)n_2 ). 
\end{aligned}\qedhere
\end{equation*}
\end{proof}

\subsection{Recursion for $\NN_{m,n}$.}\label{N-section}
The recursion relation suggested in this section (see Corollary \ref{N-recursion}) reduces computing the double sequence $\NN_{m,n}$ to a computation of 
three sequences of initial valuers, $GW_{1,n}$,  $GW^{[2]}_{1,n}$, and $\NN_{1,n}$. More precisely, this recursion relation involves also the double sequence $GW_{m,n}$ and, as additional initial values, the sequences $GW_{2,n}$,  $\NN_{2,n}$.
For the first one, $GW_{m,n}$, we have already provided a recursion relation reducing its computation to the sequences  $GW_{1,n}$ and $GW_{m,0}$. As to $\NN_{2,n}$, these numbers vanish for $n$ odd (see Theorem \ref{main-main}),
while for $n$ even we may use, as above, the formula $\NN_{2,2k}= 2^{-2}GW^{[2]}_{1,k}$ (see Theorem \ref{magic-theorem}) and consider $GW^{[2]}_{1,k}$ (instead of $\NN_{2,2k}$) as a sequence of initial valuers.

Following \cite{surgery-invariant}, we introduce the following notations:
\begin{equation}\label{definition-Gamma}
\Gamma_{m,n}=-2^{1-l}\NN_{m,n}, 
\end{equation}
where $l=\frac12(m-1-k), k\in\{0,1\}, k=m-1\,mod \, 2$;
and  
$$\Gamma_B=-2^{1-l}N_B
$$  
with $B$ written in the form $B=\frac{m}2C-\frac{n}2K'+v$ with $K'=\frac{d}2 C+K$, $v\in (C,K)^\perp\otimes\Q$.
Note that, according these definitions,
\begin{equation}\label{def-Gamma}
\Gamma_{m,n}=\sum_{B\in\LL^{m,n}_\R}\Gamma_B
\end{equation}

Recall Solomon's recursion rule ({\it cf.} \cite[(OGW3)]{Jake} and \cite[Th.1.1(RWDVV3)]{Chen})
which is a corollary of an analog for WDVV-equation designed by Solomon  \cite{open} in the framework of open strings.

\begin{theorem}\label{ThirdRelation}
Let $H_1,H_2,H_3\in H_2^-(X)$
 be fixed elements with  $H_1H_3=0$. Then, for any $B\in\lay^{m,n}_\R$ and each pair $(m,n)\in \Z_{\ge 3}\times\Z_{\ge 0}$
the following relation holds:
$$(H_1H_2)(H_3B)
\Gamma_{B}=\FS_{B}+\Theta_B^{(2)},
\text{ where }$$
\scalebox{0.9}{$
\begin{aligned}
&\FS_{B}=\frac14
\sum_{\substack {B_1+B_2=B}}
(H_1B_1)((H_3B_1)(H_2B_2)-(H_2B_1)(H_3B_2))\binom{l-1}{l_1}\binom{k}{k_1}\Gamma_{B_1}\Gamma_{B_2}\\
&\Theta_B^{(2)}=\frac12\sum _{\substack{B_F-\conj_* B_F\\+B_U=B}}\hskip-5mm
(B_UB_F)(H_1B_F)((H_3B_F)(H_2B_U)-(H_2B_F)(H_3B_U))\binom{l-1}{l_U}GW(B_F)\Gamma_{B_U}\\
& l=\frac12(m-1-k), k\in\{0,1\}, k=m-1\,mod \, 2. \qed
\end{aligned}$
}
\end{theorem}

\begin{theorem}\label{Gamma-recursion}
Let $d=K^2 \le 6$. Then, for any pair  $(m,n)\in \Z_{\ge 3}\times\Z_{\ge 1}$ and  $l=\frac12(m-1-k), k\in\{0,1\}, k=m-1\,mod \, 2$, 
we have a recursion relation
\begin{equation*}
\begin{aligned}
&2n\Gamma_{m,n}=
\frac14\sum_{\substack{n_1+n_2=n\\m_1+m_2=m\\ m_1,m_2\ge 1}}n_1(n_1m_2-n_2m_1)\binom{l-1}{ l_1}\binom{k}{k_1} \Gamma_{m_1,n_1}\Gamma_{m_2, n_2}
\\
&+ 
\frac14
\hskip-5mm
\sum_{\substack{2n_F+n_U=n\\ 2m_F+m_U=m\\ m_F, m_U\ge 1}}
\hskip-5mm
n_F(m_Un_F+m_Fn_U-\frac{dn_Un_F}2)(n_Fm_U-n_Um_F)
\binom{l-1}{ l_U} GW_{m_F,n_F}\Gamma_{m_U,n_U}.
\end{aligned}
\end{equation*}
\end{theorem}

\begin{proof} Due to invariance of the numbers $\NN_{m,n}$, it is sufficient to check the formula for
a particular real conic bundle $X$ with 
a given $d=K^2$. Let us choose a maximal one. Then, $\conj_*$ acts on $H_2(X)=\Pic X$ as multiplication by $-1$, and $(K,C)^\perp$ is isomorphic to $\DD_r$, $r=6-K^2$.

To apply Theorem \ref{ThirdRelation}, we choose $H_1=H_3=C$, $H_2=-K$. We put, as above, $K'=\frac{d}2 C+K$ and write $B$ in the form $B=\frac{m}2C-\frac{n}2K'+v$ with $v\in (C,K)^\perp\otimes\Q$.
Then,
$$
(H_1H_2)(H_3B)\Gamma_B=2n\Gamma_B
$$
while on the other side we get
\begin{equation*}
\begin{aligned}
&\FS_{B}=\frac14
\sum_{\substack {B_1+B_2=B}}
n_1(n_1m_2-m_1n_2)\binom{l-1}{l_1}\binom{k}{k_1}
\Gamma_{B_1}\Gamma_{B_2},\\
&\Theta_B^{(2)}=
\frac12
\hskip-5mm
\sum _{\substack{B_F-\conj_* B_F\\+B_U=B}}
\hskip-4mm
(\frac{m_Un_F+m_Fn_U}2-d\frac{n_Un_F}4+v_Uv_F)
(n_Fm_U-m_Fn_U)n_F
\binom{l-1}{l_U}
GW(B_F)\Gamma_{B_U}.
\end{aligned}
\end{equation*}

Since 
$\conj_*=-\id$ (due to our choice of $X$) and (by Lemma \ref{zero-sum})
$$
\sum_{B_F\in\lay^{m_F,n_F}_\C} v_F GW_{B_F}=0,
$$
the summing the above expressions for $\FS_B$ and $\Theta_B^{(2)}$ over $B\in \lay^{m,n}_\R=\lay^{m,n}_\C$ gives the result stated.
\end{proof}

Passing from $\Gamma_{m,n}$ to $\NN_{m,n}$ according  (\ref{definition-Gamma}), we get the following.

\begin{cor}\label{N-recursion}
Let $d=K^2 \le 6$. Then, for any pair  $(m,n)\in \Z_{\ge 3}\times\Z_{\ge 1}$
and $l=\frac12(m-1-k)\ge 1, k\in\{0,1\}, k=m-1\,{\rm mod} \, 2$, we have a recursion relation
\begin{equation*}
\begin{aligned}
&n \NN_{m,n}=
 \sum_{\substack{2n_F+n_U=n\\ 2m_F+m_U=m\\ m_F, m_U\ge 1}}2^{m_F-3}n_F(m_Un_F+m_Fn_U-\frac{dn_Un_F}2)(n_Fm_U-n_Um_F)\\
 &\binom{l-1}{ l_U} GW_{m_F,n_F} \NN_{m_U,n_U}
- \frac12 
\hskip-3mm
\sum_{\substack{n_1+n_2=n\\m_1+m_2=m\\ m_1,m_2\ge 1}}
\hskip-5mm
n_1(n_1m_2-n_2m_1)
\binom{l-1}{ l_1}\binom{k}{k_1} \NN_{m_1,n_1}\NN_{m_2, n_2}.
\end{aligned}
\end{equation*}
\qed
\end{cor}

\section{Up to eight blowups}\label{upto8}

Here, we find sequences $GW_{1,n}$, $GW_{1,n}^{[2]}$, and $\NN_{1,n}$  for 
conic bundles on rational surfaces with $1\le K^2\le 6$.
\footnote{Recall that the numbers $\NN_{1,n}$ are well defined only for $K^2\le 6$, while $GW_{1,n}^{[2]}$ only for $K^2\le 5$.}
Over $\C$,
such surfaces can be seen as blowups of 
$\P^2$
at $k$ points with $3\le k\le 8$.
Since each, complex or real, deformation class of conic bundles $\pi\: X\to \P^1$ has a del Pezzo representative, 
we restrict our computations to del Pezzo case.
An advantage of considering del Pezzo surfaces 
%
is based on the fact that for 
them 
 the values of $GW_{1,n}$, $GW_{1,n}^{[2]}$, and $\NN_{1,n}$  
come from the contribution of 
$(-1)$-curves,
except for the case $k=8$, $n=2$, where the contribution of 
anticanonical curves should be added.

\subsection{Class $C$ and $\DD_n$-coordinates in a blowup model}\label{lattices-prelim}
A blowup presentation $X=\P^2\#k\bar\P^2$, $k\ge1$,
yields an identification
$$
H_2(X)=H_2(\P^2)\bigoplus_{i=0}^{k-1} H_2(\bar\P^2)=\Z\oplus\Z^{k}, \text{ with standard generators } H,e_0,\dots,e_{k-1}.
$$
Considering a conic bundle with the fiber class $C=H-e_0$, we obtain a description
\[
\la K,C\ra^\perp=\{ v=yH-\sum_{i=0}^{k-1}x_ie_i \,|\, y,x_0,\dots,x_{k-1}\in\Z\, , y=
x_0 ,2y=x_1+\dots+x_{k-1}\}
\]
and an identification, $v\mapsto(x_1,\dots,x_{k-1})$, of $\la K,C\ra^\perp$ with the model of $\DD_{k-1}$ given in Section \ref{right-side}.
 

\begin{lemma}\label{Dn}
Let $X=\P^2\#k\bar\P^2$, $k\ge 1$.
Then every $h\in H_2(X)$ with $-K\cdot h=m$, $C\cdot h=n$
can be decomposed as
\begin{equation}\label{h=}
h=\frac12(mC-nK'+v),\quad  K'=K+\frac{d}2C, \ v\in \DD_{k-1}=\la K,C\ra^\perp,\ d=9-k.
\end{equation}
Conversely, 
a vector $h=\frac12(mC-nK'+v)\in H_2(X)\otimes\Q$ with
\be
 v\in \DD_{k-1}=\la K,C\ra^\perp, \frac{v}2=yH-ye_0-\sum_{i=1}^{k-1}x_ie_i
\ee
belongs to $H_2(X)$ if and only if:
\begin{enumerate}
\item $\sum_{i=1}^{k-1}x_i=n+m+\frac{nd}2\ \rm{mod}\,2$,
\item the vector $\bar x=(x_1,\dots,x_{k-1})$  lies in $\Z^{k-1}$ if $n$ is even and in $(\Z+\frac12)^{k-1}$ if $n$ is odd.
\end{enumerate}
\end{lemma}

\begin{proof}
Expressing $C,K$ and $v$ 
in the  basis formed by $H$ and $e_i$, $0\le i\le k-1$, we obtain
$$h=(\frac{m}2-\frac{nd}4)(H-e_0)+\frac{n}2(3H-\sum_{i=0}^{k-1}e_i)+\frac{v}2,\quad \frac{v}2=yH-ye_0-\sum_{i=1}^{k-1}x_ie_i,\quad
2y=\sum_{i=1}^{k-1}x_i.
$$
Integrality of the coefficients
$-\frac{n}2-x_i$ of $h$ at $e_i$, for $1\le i\le k-1$, means that $\bar x=(x_1,\dots,x_k)$ lies in $\Z^{k-1}$ for even $n$ and $(\Z+\frac12)^{k-1}$ for odd.
Integrality of the $e_0$-coefficient $\frac{nd}4-\frac{m}2-\frac{n}2-y$ is equivalent to integrality of the $H$-coefficient
$\frac{m}2-\frac{nd}4+\frac{3n}2+y$, which in its turn equivalent to
$\frac{nd}2-m-n-2y\in 2\Z$.
\end{proof}

For any $k\ge1$, $n\ge0$ we
consider a subgroup $\Cal D^{1,n}\subset \frac12 \DD_{k-1}\subset \frac12 \Z^{k-1}$ defined in terms of
$\bar x=(x_1,\dots,x_{k-1})$ and $|\bar x|^2=\sum_i x_i^2$ as 
$$
\Cal D^{1,n}=\begin{cases}
\{\bar x\in\Z^{k-1}\,|\,
|\bar x|^2=n+1-\frac{n^2d}4 \}, &\text{ if $n$ is even},\\ 
\{\bar x\in(\Z+\frac12)^{k-1}\,|\,\sum_{i=1}^{k-1}x_i\in2\Z+\frac{nd}2,
|\bar x|^2=n+1-\frac{n^2d}4 \},  &\text{ if $n$ is odd}.
\end{cases}
$$ 

\subsection{
Counting lines} 
In a del Pezzo surface $X$, by a {\it line} we mean a smooth rational curve $L\subset X$
with $L^2=-1$. By adjunction, $KL=-1$ and we
 denote by $L^{1,n}\subset\Cal L^{1,n}$ the subset formed by the classes of lines.

\begin{lemma}\label{C1n}
Let $X=\P^2\#k\bar\P^2, 1\le k\le 8,$ be a del Pezzo surface equipped with a conic bundle. Then:
\begin{enumerate}
\item
The only effective divisors whose classes belong to $\Cal L^{1,n}$, $n\ge 0$, are lines
unless $k=8$ and $n=2$. 
In the case $k=8$, 
they are the lines and the anticanonical divisors.
\item
Each line in $X$ is completely determined by its homology class.
\item
$L^{1,n}=
\{\frac{C}2-\frac{n}2K'+\frac{v}2\,|\,\frac{v}2\in\Cal D^{1,n}\}$.

\end{enumerate}
\end{lemma}

\begin{proof}
(1) and (2) are well known (see for instance \cite{TwoKinds} for enumerating of effective divisors
with $m=1$ on del Pezzo surfaces).

Relation (\ref{h=}) gives presentation of lines at level $(1,n)$ as
$L= \frac{C}2-\frac{n}2K'+\frac{v}2$. Since  $C\cdot K'=C\cdot K=-2$, $C^2=K\cdot K'=0$, we obtain
$$-1=L^2=
\frac14(4n+n^2(K')^2)+\frac{v^2}4=n-\frac{n^2d}4+\frac{v^2}4\Longrightarrow |\bar x|^2= n+1-\frac{n^2d}4.$$

By Lemma \ref{Dn}, $\bar x\in\Z^{k-1}$ for $n$ even and $\bar x\in(\Z+\frac12)^{k-1}$ for $n$ odd. Moreover, item (1) of Lemma \ref{Dn} gives
$\sum_{i=1}^{k-1}x_i=n+1+\frac{nd}2\,\rm{mod}\,2$, which for odd $n$ can be recorded as $\sum_{i=1}^{k-1}x_i\in2\Z+\frac{nd}2$.
For even $n$, item (1) is satisfied automatically:
$$\sum_{i=1}^{k-1}x_i \ \ \overset{\rm{mod}\,2}=\ \ |\bar x|^2\ \ \overset{\rm{mod}\,2}=\ \ n+1-\frac{n^2d}4\ \ \overset{\rm{mod}\,2}=\ \ n+1+\frac{nd}2.$$
It is by this reason that  it is omitted in the definition of $\Cal D^{1,n}$ for $n$ even.
\end{proof}

Let us fix $k\in\Z_{\ge1}$. Put  $d=9-k$ and consider a sequence in $n\ge0$:
\newline
\scalebox{0.95}{
\hskip-2mm$
S_{1,n}=|\Cal D^{1,n}|=
\begin{cases}
 \sharp\{ \bar x\in\Z^{k-1}\,|\ \,|\bar x|^2=n+1-\frac{n^2d}4\} &\hskip-2mm\text{ if } n \text{ is even},\\
\sharp\{ \bar x\in(\Z+\frac12) ^{k-1}\,|\ \, \sum\limits_{i=1}^{k-1}x_i\in2\Z+\frac{nd}2,
|\bar x|^2=n+1-\frac{n^2d}4\} &\hskip-2mm\text{ if } n \text{ is odd}.
\end{cases}
$
}
\begin{remark}\label{half}
Note that for odd $n$ the condition $\sum_{i=1}^{k-1}x_i\in2\Z+\frac{nd}2$
is satisfied precisely for a half of vectors
$\bar x\in(\Z+\frac12) ^{k-1}$ satisfying condition $|\bar x|^2=n+1-\frac{n^2d}4$.
This is because replacing $x_1$ by $-x_1$ changes the pairity of $\sum_{i=1}^{k-1}x_i$ while not changing $|\bar x|^2$.
Therefore, we may equivalently define
$$S_{1,n}=|\Cal D^{1,n}|=
\frac12\sharp\{ \bar x\in(\Z+\frac12) ^{k-1}\,|\ \, 
|\bar x|^2=n+1-\frac{n^2d}4\} \ \ \text{ if } n \text{ is odd}.
$$
\end{remark}

The following proposition is a direct consequence of Lemma \ref{C1n}.

\begin{proposition}\label{dP-count}
For a conic bundle on a del Pezzo surface $X=\P^2\#k\bar\P^2, 1\le k\le 8$:

{\rm (1) } 
The number of lines of level $(1,n)$ is equal to $S_{1,n}$.

{\rm(2)} For $k\le 7$ as well as for $k=8$, $n\ne2$,
 we have $GW_{1,n}=S_{1,n}$,
 while for $k=8$ we have $GW_{1,2}=S_{1,2}+GW(-K)$.
 \qed
\end{proposition}

\begin{proposition}\label{gf-GW}
For  a conic bundle on $X=\P^2\#k\bar\P^2, 2\le k\le 8$:
$$\sum_{n\ge0} GW_{1,n}q^{n+1}=
\begin{cases}
2(k-1)q+2^{k-2}q^2&\text{ for } 2\le k\le5,\\
10q+16q^2+q^3 &\text{ for } k=6,\\
12q+32q^2+12q^3 &\text{ for } k=7,\\
14q+64q^2+96q^3+64q^4+14q^5 &\text{ for } k=8.\\
\end{cases}$$
\end{proposition}

\begin{proof}
It follows from  Proposition \ref{dP-count} 
that $GW_{1,0}$ counts unit vectors
in $\Z^{k-1}$, which gives $2(k-1)$, while $GW_{1,1}$ counts with coefficient $\frac12$
the vectors
of square $\frac{k-1}4$ in $(\frac12+\Z)^{k-1}$, which gives $2^{k-2}$ (see Table \ref{dP-constants}).
For $k=6$, the value $GW_{1,2}=S_{1,2}=1$ counts the zero element in $\Z^5$. 
The symmetry in the last two rows of Table \ref{dP-constants} shows that 
$GW_{1,2}=GW_{1,0}=14$ for $k=7$, while
$GW_{1,4}=GW_{1,0}=2$ and $GW_{1,3}=GW_{1,1}=64$ for  $k=8$.
The last remaining value $GW_{1,2}=4\binom72+12=96$ results from counting vectors $\bar x\in\Z^7$, $|\bar x|^2=2$, which gives $S_{1,2}$
to which we add the number $GW(-K)=12$ of anti-canonical rational curves.
\end{proof}
\begin{table}[h!]
\caption{Values of $n+1-\frac{n^2d}4$. Sign "-" stands for negative values.}\label{dP-constants}
\begin{tabular}{cccccc}
\hline
k\textbackslash n&0&1&2&3&4\\
\hline
2&1&$\frac14$&-&-&-\\
3&1&$\frac12$&-&-&-\\
4&1&$\frac34$&-&-&-\\
5&1&1&-&-&-\\
6&1&$\frac54$&0&-&-\\
7&1&$\frac32$&1&-&-\\
8&1&$\frac74$&2&$\frac74$&1\\
\hline
\end{tabular}
\end{table}

\subsection{
Generating function for the moments}
For a fixed  $k\in\Z_{\ge1}$, put $d=9-k$ and consider a sequence in  $n\in\Z_{\ge0}$:
$S_{1,n}^{[2]}=2\sum x_1^2$,  with
summation 
over the set
\[
\begin{cases}
\{\bar x\in\Z^{k-1} \,|\ \ |\bar x|^2=n+1-\frac{n^2d}4\}, \text{ if }n\ \text{ is even}, \\
\{\bar x\in(\Z+\frac12)^{k-1}\,|\  \sum_{i=1}^{k-1} x_i=\frac{nd}2\,{\rm mod}\,2, |\bar x|^2=n+1-\frac{n^2d}4\}
\text{ if }n \text{ is odd}.
\end{cases}
\]
Note that for odd $n$ the definition can be reformulated (see Remark \ref{half}):
$$
S_{1,n}^{[2]}=\sum x_1^2, \text{ summed over the set }
\{\bar x\in(\Z+\frac12)^{k-1}\,|\ |\bar x|^2=n+1-\frac{n^2d}4\}.
$$

\begin{proposition}
For $X=\P^2\#k\bar\P^2, 4\le k\le 8$ endowed with a conic bundle:
 \begin{enumerate}\item
 $GW_{1,n}^{[2]}=S_{1,n}^{[2]}$.
 \item
The  generating function for 
 $GW_{1,n}^{[2]}$ is
$$\sum_{n\ge0} GW_{1,n}^{[2]}q^{n+1}=
\begin{cases}
4q+2^{k-3}q^2 &\text{ for } 4\le k\le 6,\\
4q+16q^2+4q^3 &\text{ for } k=7,\\
4q+32q^2+48q^3+32q^4+4q^5 &\text{ for } k=8.\\
\end{cases}$$
\end{enumerate}
\end{proposition}

\begin{proof}
By Proposition \ref{root-moments}, for counting the moments we can use a vector
$v=(1,0,\dots,0)$ instead of a root $e$ and then multiply the $v$-moment by the normalization factor $\frac{e^2}{v^2}=2$.
This gives a sum $2\sum x_1^2$ over the set $\Cal D^{1,n}$ from Lemma \ref{C1n}(3), which proves the claim for even $n$.
In the case of odd $n$, we drop the condition on $\sum_{i=1}^{k-1}x_i$ for $\Cal D^{1,n}$ in Lemma \ref{C1n}(3)
and preserve only the condition $|\bar x|^2=n+1-\frac{n^2d}4$. This doubles the number of vectors  with the same value $|x_1|$ (see the proof of Proposition  \ref{dP-count}(1)),
and therefore we drop factor $2$ (which appeared for even $n$).

Finally, we $GW_{1,n}^{[2]}=S_{1,n}^{[2]}$,  using Proposition \ref{dP-count}, since
in the case $d=1$, $n=2$ the curves in $|-K|$ contribute 0 to $GW_{1,2}^{[2]}$.
\end{proof}

\begin{remark}\label{palyndrom}
The numbers $GW_{m,n}$, $GW_{m,n}^{[2]}$ and $\NN_{m,n}$
have the following symmetry properties if $d=1,2,4$:
$$
GW_{m,n}=GW_{m,\frac4dm-n},\qquad GW_{m,n}^{[2]}=GW_{m,\frac4dm-n}^{[2]},\qquad \NN_{m,n}=\NN_{m,\frac4dm-n}
$$
The reason is existence of an involution on $X$ for $d=1,2,4$ (namely, Bertini, Geiser and one of  basic involutions for $d=4$, see \cite[Sect. 6.4]{dolg-isk})
which transforms a conic bundle whose fiber class is $C$ to a conic bundle with the fiber class $C'=\frac4d(-K)-C$.
A divisor class 
$\a\in\Cal L^{m,n}$ with respect to $C$, belongs to
$\Cal L^{m,\frac4dm-n}$ with respect to $C'$.
On the other hand, the count for each of the indicated numbers is independent of the choice of a conic bundle.
\end{remark}

\subsection{Numbers $\NN_{1,n}$}

\begin{proposition}
The only non-zero values of $\NN_{1,n}$ for 
surfaces $X=\P^2\#k\bar\P^2$ with $3\le k\le 8$ are
$\begin{cases}
\NN_{1,0}=2, &\text{ for } k=4,\\
\NN_{1,0}=2, \NN_{1,2}=1&\text{ for } k=6,\\
\NN_{1,0}=2, \NN_{1,2}=4, \NN_{1,4}=2&\text{ for } k=8.\\
\end{cases}$
\end{proposition}

\begin{proof}
In accordance with Theorem \ref{main-main}, $\NN_{1,n}=0$ if $n$ or $k$ is odd (since $s$ and $k$ have opposite parity).
As it follows, for example, from Propositions \ref{dP-count}-\ref{gf-GW}, in
the cases $k=4, 6,8$ the numbers $\NN_{1,n}$ vanish for $n\ge2$, $n\ge3$ and $n\ge5$ respectively,
which shows that the only non-vanishing $\NN_{1,n}$ are the indicated ones.

For proving $\NN_{1,0}=2$ in the case  $k=2r\in\{4,6,8\}$, consider a real model of $X$ obtained by blowing up  $r$ generic pairs of conjugate points,
$p_1,\bar p_1,\dots,p_r,\bar p_r$, so that $X_\R=\P^2_\R$. 
Equip $X$ with a conic bundle obtained by taking pull-back to $X$ of the pencil of conics passing through
$p_1,\bar p_1,p_2,\bar p_2$. Denote by $L_i$ the 
in $X$ obtained by taking pull-back 
of the line
in $\P^2$ passing through $p_i$ and $\bar p_i$. The lines $L_1,L_2$ are the only real ones on level $(1,0)$, and each of them contributes $+1$ to $\NN_{1,0}$ with respect to the monic $\Pin^-$-structure on $\P^2_\R$.

In the case $k=6$, the lines $L_1,L_2,L_3$ are the only real ones and only $L_3$ 
has level $(1,2)$. This gives $\NN_{1,2}=1$.

In the case $k=8$, the symmetry indicated in Remark \ref{palyndrom} implies $\NN_{1,4}=\NN_{1,0}=2$.
A similar straightforward count gives also $\NN_{1,2}=4$, but it looks smarter to use
$\sum_n\NN_{1,n}=\NN_1$ (see Theorem 
\ref{refinement}) and $\NN_1=8$ (see \cite{TwoKinds}).
\end{proof}

\section{Nine blowups}\label{nine}
In this section, we consider rational conic bundles with $K^2 = 0$ and provide a full set of initial values and bivariate recursive relations for the invariants $N_{m,n}$.

Under assumption $K^2=0$, Theorem \ref{magic-theorem} shows that $N_{2m,2n}= 2^{m-3}GW^{[2]}_{m,n}$ 
for any $m\in\Z_{\ge 1}, n\in\Z_{\ge 0}$, while
according to Theorem \ref{main-main} all the other numbers $N_{*,*}$ are zero. This reduces computing of the double sequence $N_{*,*}$ to that of $GW^{[2]}_{*,*}$. On the other hand,
by substituting $d=0$ into Theorems \ref{simple-recursion-any-d} and \ref{moment-recursion-any-d}
we get a bivariate recursion reducing computation of $GW_{m,n}$ and $GW^{[2]}_{m,n}$ with $m\ge 2$ to computation of the 
initial values, $GW_{m,0}$, $GW_{1,n}$ and $GW^{[2]}_{1,n}$.

Below, we accomplish the latter task in
two steps.
First, we find a generating function for $GW_{1,n}$,
and then we derive from it a generating function for $GW^{[2]}_{1,n}$.
As to $GW_{m,0}$, their values are already given in Proposition \ref{GWm0}.

Since Gromov–Witten invariants are preserved under deformation, it suffices to perform these calculations for a particular choice of nine points.
Following J.~Bryan and N.C.~Leung \cite{BL} we choose nine blow-up points to be the base points of a pencil of plane cubics. Then, the blown-up surface, which we continuer to denote by $X$, becomes equipped with an elliptic fibration $f: X\to \P^1$ with fibers that form the anticanonical linear system.

\subsection{
Generating function for $GW_{1,n}$}\label{gener-simple}
First, let us recall some basic results.

\begin{proposition}{\cite[§4]{Manin},\cite[Lemma 12 and Corollary 13]{M-P}}\label{section-criter}
Let $f\: X\to \P^1$ be a rational relatively minimal elliptic surface with a section and without reducible fibers. Then:
\begin{enumerate} 
\item
A homology class $\alpha\in H_2(X)$ can be realized by a section, if and only if $\alpha^2=
\alpha  K=-1$.
\item Each section is uniquely determined by its homology class.
\item If $D\subset X$ is a rational curve with $-DK=1$, then
the splitting of $D$ into irreducible components has form $D=L+\sum m_iF_i$, $m_i\in\Z_{\ge 0},$ where $L$ is a section and $F_i$ are the singular fibers of $f$.
\qed
\end{enumerate}
\end{proposition}

Through the rest of this section we use notations $S_{1,n}, S^{[2]}_{1,n}$ introduced in Section \ref{upto8} which we specialize to the case $d=K^2=0$. In particular,
this gives
\begin{equation}\label{def-S}
S_{1,n}=
\begin{cases}
 &\sharp\{ \bar x\in\Z^{8}\,|\ \,|\bar x|^2=n+1\},\ \text{ for even } n,\\
& \sharp\{ \bar x\in(\Z+\frac12) ^{8}\,|\ \, |\bar x|^2=n+1, \, \sum x_i=0 \,{\rm mod}\, 2\} \text{ for odd } n.
\end{cases}
\end{equation}

To compute 
$GW_{1,n}$, we find first 
the cardinality $|L^{1,n}|$
of the subset 
$L^{1,n}\subset \Cal L^{1,n}$ formed by classes
realized by sections.

\begin{cor}\label{S=L}
If an elliptic surface as in Proposition \ref{section-criter} is endowed
with a conic bundle, then
$|L^{1,n}|=S_{1,n}$,
for any $n\ge 0$.
\end{cor}
\begin{proof} 
It follows immediately from Proposition \ref{section-criter} and Lemma \ref{Dn} applied to $X=\P^2\# 9\bar\P^2$, since in terms of decomposing $\alpha$ in a form $\alpha=\frac12(C-nK+v)$ with $v\in \DD_8=\la C,K\ra^\perp$, $\frac{v}2=\bar x$
we get  $\alpha^2=-1$ if and only if  $|\bar x|^2= n+1$.
\end{proof}

One can deduce from \cite{conway}[Sec.8.1] a formula
$$
S_{1,n}=16[\s_3(n)-\s_3(\frac{n}2)],\quad \text{where } \s_3(n)=\sum_{d|n}d^3, \text{ and } \s_3(n)=0 \text{ if } n\notin \Z,
$$
with the generating function in $q=e^{\pi i z}$ expressed through theta-series
$$\sum_{n\ge1} S_{1,n}q^n=
\frac12(\theta_2(z)^8+\theta_3(z)^8-\theta_4(z)^8)=
16q+128q^2+448q^3+1024 q^4+\dots
$$

We give details on a slightly different way of presentation.

\begin{proposition}\label{lines-series}
\begin{equation}\label{identity-1}
\begin{aligned}
\sum_{k=0}^\infty S_{1,2k}q^{2k+1}= & \frac12(\theta^8_3(q) - \theta^8_3(-q)),\\
\sum_{k=0}^\infty S_{1,2k+1}q^{2k+2}= & \frac12
\theta^8_2(0,z),
\end{aligned}
\end{equation}
where
\begin{equation}\label{theta-defin}
\theta_3(q)=\sum_{m=-\infty}^{+\infty} q^{m^2}, \,\, \theta_2(\tau, z)=\sum_{m=-\infty}^{+\infty} e^{(2m+1 ) i\tau + \pi i z (m+\frac12)^2}
\end{equation}
with $q=e^{\pi i z}$.
\end{proposition}

\begin{proof} According to (\ref{theta-defin})
 we have $\frac12(\theta^8_3(q) - \theta^8_3(-q))= \sum_{k=0}^\infty c_{2k+1}q^{2k+1}$ where
$c_{2k+1}=\sharp\{ \bar x\in\Z^8 \,|\,\sum x_i^2=2k+1\}$. Thus, the first identity in (\ref{identity-1}) follows from 
the first part of (\ref{def-S}).

Similarly, we have
$$
\frac12 \theta^8_2(0,z)
= \sum_{\bar m\in (\Z+\frac12)^8}\frac12 
q^{\sum m_j^2} =
 \sum_{k=0}^\infty  \frac12 c_{2k+2}q^{2k+2}
$$
where $c_{2k+2}=\sharp\{ \bar m\in (\Z+\frac12)^8 \,|\, \sum m_i^2=2k+2\}$.
Therefore, the second identity in  (\ref{identity-1}) follows from 
the second
part of (\ref{def-S}) and Remark \ref{half}.
\end{proof}

\begin{theorem}{\cite[Theorem 6.2]{BL}}\label{bl-theorem}  Let $f \: X\to \P^1$ be a rational relatively minimal elliptic surface with a section $L$ and without reducible fibers. Then,
\begin{equation}\hskip24mm
\sum_{r\ge 0}
GW(L-rK)q^r=\prod_{j\ge 1}  (1-q^j)^{-12}.
\hskip25mm \qed \end{equation}
\end{theorem}

\begin{theorem}
\begin{equation}\label{simple-sum}
\sum^\infty_{n=0} GW_{1,n} q^{n+1}=\Cal F(q)\cdot\Pi_{j=1}^\infty (1-q^{2j})^{-12}
\end{equation}
where
\begin{equation}
\Cal F(q)= \frac12(\theta^8_3(q) - \theta^8_3(-q)+\theta^8_2(0,z)).
\end{equation}
\end{theorem}
\begin{proof}
For $L-rK$ as in 
Theorem \ref{bl-theorem}, the $C$-level of $L-rK$ is equal to $C\cdot L +2r$. Since in addition $(L-rK)^2=2r-1$, 
Proposition \ref{section-criter}
implies that 
the set of divisor classes of rational curves of bi-level $(1,n)$
splits into a disjoint union of the sets 
$L^{1,n-2r}-rK=\{\a-rK\,|\,\a\in L^{1,n-2r}\}$, ${r\ge0}$,
which implies
$$
GW_{1,n}=S_{1,n}GW(L)+ S_{1,n-2}GW(L-K)+S_{1,n-4}GW(L-2K)+\dots
$$
Combining this with Theorem \ref{bl-theorem} we get (\ref{simple-sum}).
\end{proof}

\subsection{Generating function for $GW^{[2]}_{1,n}$}
By analogy with Section \ref{gener-simple}, we start with counting the partial moments $S^{[2]}_{1,n}$ 
for sections 
of bi-level $(1,n)$:
$$
S^{[2]}_{1,n}:=\sum_{\alpha\in L_{1,n}}
(e\alpha)^2 GW(\alpha).
$$
For that, we use once more 
$S_{1,n}=\vert L_{1,n}\vert$, $GW(\a)=1$
and, in accordance with Proposition \ref{root-moments},  get the following expression by replacing $e$ with a vector $(1, 0, \dots, 0)$:
\begin{equation}\label{renormalization}
S^{[2]}_{1,n}= 2 \sum x_1^2
\,\,\text{taken over}\,\, 
\begin{cases}
\bar x\in\Z^8 : \sum x_i^2=n+1\,\text{for $n$ even,} \\
\bar x\in(\Z+\frac12)^8 : \sum x_i^2=n+1, \sum x_i=0\, ({\rm mod}\,2)\, \text{for $n$ odd.}
\end{cases}
\end{equation}

\begin{proposition}\label{moment-series}
\begin{equation}\label{identity-2}
\begin{aligned}
\sum_{k=0}^\infty S^{[2]}_{1,2k}q^{2k+1}=
& \frac{1}{8\pi i}\big(\frac{\partial \theta^8_3}{\partial z} (z)-\frac{\partial \theta^8_3}{\partial z} (z+1)\big),
\\
\sum_{k=0}^\infty S^{[2]}_{1,2k+1}q^{2k+2}= 
& \frac{1}{8\pi i}\frac{\partial \theta^8_2}{\partial z} (0,z),
\end{aligned}
\end{equation}
where $q=e^{\pi i z}$
and $\theta_2(\tau,z)$ is as defined in Proposition \ref{lines-series}.
\end{proposition}
\begin{proof} By definition,
\begin{equation*}
\begin{aligned}
\frac{\partial \theta^8_3}{\partial z}(z)=& 8\pi i\sum_{\bar m\in \Z^8}m_1^2q^{\sum m_j^2}, \\
\frac{\partial \theta^8_3}{\partial z}(z+1)=& 8\pi i\sum_{\bar m\in \Z^8}m_1^2(-q)^{\sum m_j^2}, \\
\frac{\partial\theta^8_2}{\partial z}(0,z)=&  8\pi i \sum_{\bar m\in \Z^8}(m_1+\frac12)^2 q^{\sum( m_j+\frac12)^2}.
\end{aligned}
\end{equation*}
Being combined with (\ref{renormalization}) and Remark \ref{half} this gives the stated result.
\end{proof}

\begin{theorem}
\begin{equation}\label{moment-sum}
\sum^\infty_{n=0} GW^{[2]}_{1,n} q^{n+1}=\Cal F^{[2]}(q)\cdot\Pi_{k=1}^\infty (1-q^{2k})^{-12}
\end{equation}
where
\begin{equation}
\Cal F^{[2]}(q)=  \frac{1}{8\pi i}\big(\frac{\partial \theta^8_3}{\partial z} (z)-\frac{\partial \theta^8_3}{\partial z} (z+1)+\frac{\partial \theta^8_2}{\partial z} (0,z)\big)
\end{equation}
with $q=e^{\pi i z}$.
\end{theorem}
\begin{proof} This follows from Proposition \ref{moment-series} and a decomposition
\begin{equation*}
GW^{[2]}_{1,n}=S^{[2]}_{1,n}+ S^{[2]}_{1,n-2}GW(L-K)+S^{[2]}_{1,n-4}GW(L-2K)+\dots
\end{equation*}
where $L-nK$ is as in 
Theorem \ref{bl-theorem}.
\end{proof}

A few first terms of the sequences  $GW_{1,n}$ and $GW_{1,n}^{[2]}=4\NN_{2,2n}$ (computed using Theorems \ref{simple-sum}, \ref{moment-sum})
are as follows.

 \begin{table}[h!]
 \caption{First values of $GW_{1,n}$ and $GW_{1,n}^{[2]}=4\NN_{2,2n}$}
\begin{tabular}{|ccccccc|}
\hline
{$GW_{1,0}$}&{$GW_{1,1}$}&{$GW_{1,2}$}&{$GW_{1,3}$}&{$GW_{1,4}$}&{$GW_{1,5}$}&{$GW_{1,6}$}\\
\hline
16&128&640&2560&8832&27392&78336\\
\hline
\end{tabular}
\vskip2mm 
\begin{tabular}{|ccccccc|}
\hline
{$GW_{1,0}^{[2]}$}&{$GW_{1,1}^{[2]}$}&{$GW_{1,2}^{[2]}$}&{$GW_{1,3}^{[2]}$}&{$GW_{1,4}^{[2]}$}&{$GW_{1,5}^{[2]}$}&{$GW_{1,6}^{[2]}$}\\
\hline
4&64&384&1792&6912&23712&50832\\
\hline
\hline
$\Cal N_{2,0}$&$\Cal N_{2,2}$&$\Cal N_{2,4}$&$\Cal N_{2,6}$&$\Cal N_{2,8}$&$\Cal N_{2,10}$&$\Cal N_{2,12}$\\
\hline
1&16&96&898&3456&5928&12708\\
\hline
\end{tabular}
\end{table}

\section{Concluding Remarks}\label{concluding}

\subsection{An additional recursion formula for $GW_{m,n}$.}
The following proposition is one more straightforward consequence of the associativity law for the small quantum product.
\begin{proposition}\label{additional-prop-associativity}
Let $X$ be a rational surface. Then, for any $H_1, H_2, H_3 \in H_2(X)$ and any $\alpha \in H_2(X)$ with $-K\alpha\ge 3$, the following relation holds:
\begin{equation}\label{additional-associativity}
\begin{aligned}
\bigl((H_1H_2)(\alpha H_3) - (H_2H_3)(\alpha H_1)\bigr)GW(\alpha) = &\\
\sum_{\alpha=\alpha_1+\alpha_2} GW(\alpha_1)GW(\alpha_2)(\alpha_1\alpha_2)\bigl( \binom{m-3}{m_2-1}(\alpha_1H_1)(\alpha_2H_2)(\alpha_2H_3)- & \\
\binom{m-3}{m_1-1}(\alpha_1H_1)(\alpha_1H_2)(\alpha_2H_3)\bigr) &\\
\end{aligned}
\end{equation}
where $m=-K\alpha$ and $m_i=-K\alpha_i$, $i=1,2$.\qed
\end{proposition}

Taking $H_1=K, H_2=H_3=C$ in Proposition \ref{additional-prop-associativity} we get the following recursion statement.

\begin{theorem}\label{additional-recursion-any-d} For any $m\ge 3$, the following relation holds
\begin{equation}
\begin{aligned}
& 2n\,GW_{m,n}=
\sum_{\substack{m=m_1+m_2 \\ n=n_1+n_2}}GW_{m_1,n_1} GW_{m_2,n_2}\bigl(\binom{m-3}{m_1-1} m_1n_1n_2 -  \\
&\binom{m-3}{m_2-1} m_1 n_2^2\bigl)
(\frac14 n_1n_2d-\frac12(m_1n_2+n_1m_2)).
\end{aligned}
\end{equation}
\end{theorem}

\subsection{An additional recursion formula for $GW^{[2]}_{m,n}$.}
The idea is to apply twice 
Proposition \ref{prop-associativity} and follow the same simplification arguments
as  in the proof of Theorem \ref{moment-recursion-any-d}: first, to $H_1= H_3=-K'$,
$H_2=H_4=C-e$; second, to $H'_1=H'_3=-K$, $H'_2=H'_4=C-e$.
Then, we take the sum of the results (which eliminates $d$ from the coefficient 
before $GW^{[2]}_{m,n}$). As an outcome, we get the following recursion having a non-zero constant coefficient in front of $GW^{[2]}_{m,n}$.

\begin{theorem}\label{better-moment-recursion}
Let $K^2\le 5$.
Then,
for any $m\ge 2, n\ge 0$ the following recursion relation holds
\begin{equation}
\begin{aligned}
&(2(m-\frac{nd}2)^2+2m^2+4n^2)GW_{m,n}+8GW^{[2]}_{m,n}=
\sum_{\substack{m=m_1+m_2\\ n=n_1+n_2}}\big(a\,GW_{m_1,n_1}GW_{m_2,n_2}+\\
&b_1GW^{[2]}_{m_1,n_1}GW_{m_2,n_2}+b_2 GW_{m_1,n_1}GW^{[2]}_{m_2,n_2} +c \, GW^{[2]}_{m_1,n_1}GW^{[2]}_{m_2,n_2}\big) \\
\end{aligned}
\end{equation}
where
$$
\begin{aligned}
 &a= \binom{m-2}{m_1-1} (\frac12(m_1n_2+m_2n_1) - \frac{n_1n_2d}4)((m_1-\frac{n_1d}2)^2n_2^2 +m_1^2n_2^2- \\
 &-m_1m_2n_1n_2 -(m_1-\frac{n_1d}2)(m_2-\frac{n_2d}2)n_1n_2 ), \\
& b_1= \binom{m-2}{m_1-1}(\frac12(m_1n_2+m_2n_1)- \frac{n_1n_2d}4)(2n_2^2-( 2m_2-\frac{n_2d}2)n_2), \\
&b_2=\binom{m-2}{m_1-1} (\frac12(m_1n_2+m_2n_1)- \frac{n_1n_2d}4)((m_1-\frac{n_1d}2)^2+m_1^2 - (2m_1-\frac{n_1d}2)n_1), \\
&c= \frac12 \binom{m-2}{m_1-1} ((m_1-\frac{n_1d}2)(m_2-\frac{n_2d}2)+n_1n_2+n_1(m_2-\frac{n_2d}2)+\\
&(m_1+n_1)(m_2+n_2)--3(m_1-\frac{n_1d}2)n_2 -4m_1n_2 ).  \\
\end{aligned}
$$
\end{theorem}

\subsection{Tables}
The values $GW_{1.n}$, $\NN_{1,n}$, and $GW_{1,n}^{[2]}$ given in the tables below
for $X=\P^2\#k\bar\P^2$, $k\le8$, were found in Section \ref{upto8}.

For $k\le6$, the 
classes of effective divisors of bi-level $(2,n)$
are limited to
those of fiber-classes of conic bundles,
which reduces their enumeration to a purely arithmetical task,
and, thus, we can find $GW_{m.n}$ and $GW_{m,n}^{[2]}$ for $m=2$ straightforwardly, as we did it for $m=1$.
By Theorem \ref{main-main},  $\NN_{2,n}=0$ for odd $n$.
For even $n$ and $k\ge 5$, in accordance with Theorem \ref{magic-theorem} one can use $\NN_{2,n}=\frac14 GW_{1,\frac{n}2}^{[2]}$.
For even $n$ and $k=3,4$, the values $\NN_{2,n}$ can be easily found
taking as example a conic bundle $X\to \P^1$ having minimal number of real singular fibers.
Then, $X_\R$ is $S^2$ if $k=3$ (with 
its  unique $\Pin^-$-structure),
and $\Rp2$ if $k=4$ (with 
its
unique monic $\Pin^-$-structure), and, thus,
all 
effective divisor classes of bi-level $(2,n)$
are counted with weight $+1$.

In the case $k=7$, the computations are analogous, except for the layer $\Cal L^{2,2}$,
which contains the
class $-K$ in addition to the fiber-classes of conic bundles.
This unique additional effective divisor class, $-K$, contributes $12$ to $GW_{2.2}$ and $0$ to $GW_{2,2}^{[2]}$.
Similar to above cases, $\NN_{2,2}=\frac14 GW_{1,1}^{[2]}=4$.

In the case $k=8$,
we apply the recurrence relations
from Theorems \ref{simple-recursion-any-d} and \ref{moment-recursion-any-d}) to find $GW_{m.n}$ and $GW_{m,n}^{[2]}$ for $m=2$ and $3$,
and then use Theorem \ref{magic-theorem} to find 
$\NN_{2,2n}=\frac14 GW^{[2]}_{1,n}$.
By Theorem \ref{main-main}, $\NN_{2,2n+1}=0$.
The values of $\NN_{3,n}$ for $k\ge 3$ are found using Corollary \ref{N-recursion}.

\begin{remark}
The condition $s\ge4$ in Theorem \ref{magic-theorem} turns out to be essential. Namely,
 for $k=4$ (that is, $s=3$) the claim of that theorem is not satisfied for instance for $N_{2,2}=0$, which is not $\frac14GW_{1,1}^{[2]}=2$.
\end{remark}

In Tables \ref{below8}--\ref{k=8}
below we let
$$
GW_m:=\sum_nGW_{m,n}, \quad GW^{[2]}_m:=\sum_nGW^{[2]}_{m,n}.
$$

\begin{table}[h!]
\caption{The values on the  levels $m=1,2$ for $k\le7$}\label{below8}
\hskip-15mm
\begin{tabular}{|c|ccc|c||ccccc|c|}
\hline
k&\scalebox{0.6}{$GW_{1,0}/\NN_{1,0}$}&\scalebox{0.60}{$GW_{1,1}/\NN_{1,1}$}&\scalebox{0.60}{$GW_{1,2}/\NN_{1,2}$}
&\scalebox{0.6}{$GW_1/\NN_{1}$}&
\scalebox{0.6}{$GW_{2,0}/\NN_{2,0}$}&\scalebox{0.6}{$GW_{2,1}/\NN_{2,1}$}&\scalebox{0.6}{$GW_{2,2}/\NN_{2,2}$}&\scalebox{0.7}{$GW_{2,3}/\NN_{2,3}$}&\scalebox{0.6}{$GW_{2,4}/\NN_{2,4}$}
&\scalebox{0.6}{$GW_2/\NN_{2}$}\\
\hline
3&4/0&2/0&0/0&6/0&1/1&2/0&0/0&0/0&0/0&3/1\\
4&6/2&4/0&0/0&\scalebox{0.9}{10/2}&1/1&4/0&0/0&0/0&0/0&5/1\\
5&8/0&8/0&0/0&\scalebox{0.9}{16/0}&1/1&8/0&1/1&0/0&0/0&10/2\\
6&\scalebox{0.9}{10/2}&\scalebox{0.9}{16/0}&1/1&\scalebox{0.9}{27/3}&1/1&16/0&10/2&0/0&0/0&27/3\\
7&\scalebox{0.9}{12/0}&\scalebox{0.9}{32/0}&\scalebox{0.8}{12/0}&\scalebox{0.9}{56/0}&1/1&32/0&72/4&32/0&1/1&\scalebox{0.9}{138/6}\\
\hline
\end{tabular}

\hbox{
\hskip-1mm\begin{tabular}{|c|ccc|c|}
\hline
k&\scalebox{0.85}{$GW^{[2]}_{1,0}$}&\scalebox{0.85}{$GW^{[2]}_{1,1}$}&\scalebox{0.85}{$GW^{[2]}_{1,2}$}&
\scalebox{0.85}{$GW^{[2]}_1$}\\
\hline
4&\phantom{Aa}4\phantom{Aa}&\phantom{Aa}2\phantom{Aa}&\phantom{aa-}0\phantom{aa-}&\phantom{A}6\phantom{A}\\
5&4&4&0&8\\
6&4&8&0&12\\
7&4&16&4&24\\
\hline
\end{tabular}
\begin{tabular}{|ccccc|c|}
\hline
  \phantom{a}\scalebox{0.85}{$GW^{[2]}_{2,0}$}\phantom{a}&\phantom{}\scalebox{0.85}{$GW^{[2]}_{2,1}$}\phantom{a}&\phantom{}\scalebox{0.85}{$GW^{[2]}_{2,2}$}\phantom{a} &\phantom{a}\scalebox{0.85}{$GW^{[2]}_{2,3}$}\phantom{a}&
 \phantom{-}\scalebox{0.85}{$GW^{[2]}_{2,4}$}\phantom{a}&
\phantom{}\scalebox{0.85}{$GW^{[2]}_2$}\\
\hline
0&\phantom{Aa}2\phantom{Aa}&\phantom{Aa}0\phantom{Aa}&\phantom{Aa}0\phantom{Aa}&\phantom{A}0\phantom{A}&\phantom{A}2\phantom{A}\\
0&4&0&0&0&4\\
0&8&4&0&0&12\\
0&16&40&16&0&72\\
\hline
\end{tabular}
}
\end{table}

\begin{table}[h!]
\caption{The values on the level $m=3$  for $k\le7$}
\begin{tabular}{|c|cccccc||c|}
\hline
k&\scalebox{0.8}{$GW_{3,0}/\NN_{3,0}$}&\scalebox{0.8}{$GW_{3,1}/\NN_{3,1}$}&\scalebox{0.8}{$GW_{3,2}/\NN_{3,2}$}&\scalebox{0.8}{$GW_{3,3}/\NN_{3,3}$}&\scalebox{0.8}{$GW_{3,4}/\NN_{3,4}$}
&\scalebox{0.8}{$GW_{3,5}/\NN_{3,5}$}
&\scalebox{0.8}{$GW_3/\NN_3$}\\
\hline
3&0/0&2/0&0/0&0/0&0/0&0/0&2/0\\
4&0/0&4/0&1/1&0/0&0/0&0/0&5/1\\
5&0/0&8/0&8/0&0/0&0/0&0/0&16/0\\
6&0/0&16/0&52/2&16/0&0/0&0/0&84/2\\
7&0/0&32/0&304/0&576/0&304/0&32/0&1248/0\\
\hline
\end{tabular}
\end{table}

\begin{table}[h!]
\caption{The values on the levels $m=1,2$ and $3$ for $k=8$}\label{k=8}
\hskip-40mm
\vbox{\begin{tabular}{|c|ccccccccc||c|}
\hline
\scalebox{0.8}{n}&\scalebox{0.8}{$0$}&\scalebox{0.8}{$1$}&\scalebox{0.8}{$2$}&\scalebox{0.8}{$3$}&\scalebox{0.8}{$4$}&\scalebox{0.8}{$5$}&\scalebox{0.8}{$6$}&\scalebox{0.8}{$7$}&\scalebox{0.8}{$8$}&\scalebox{0.8}{the sum}\\
\hline
\scalebox{0.8}{$GW_{1,n}/GW_{1,n}^{[2]}/\NN_{1,n}$}&
\scalebox{0.8}{14/4/2}&\scalebox{0.8}{64/32/0}&\scalebox{0.8}{96/48/4}&\scalebox{0.8}{64/32/0}&\scalebox{0.8}{14/4/2}&\scalebox{0.8}{0/0/0}&\scalebox{0.8}{0/0/0}&\scalebox{0.8}{0/0/0}&\scalebox{0.8}{0/0/0}&\scalebox{0.8}{252/120/8}\\
\scalebox{0.70}{$GW_{2,n}/GW_{2,n}^{[2]}/\NN_{2,n}$}&
\scalebox{0.7}{1/0/1}&\scalebox{0.7}{64/32/0}&\scalebox{0.7}{$448/288/8$}&\scalebox{0.7}{$1216/864/0$}&\scalebox{0.7}{1672/1232/12}
&\scalebox{0.7}{$1216/864/0$}&\scalebox{0.7}{$448/288/8$}&\scalebox{0.7}{64/320}&\scalebox{0.7}{1/0/1}&\scalebox{0.7}{5130/3600/30}\\
\hline
\end{tabular}}
\vskip3mm
\hskip-37mm\vbox{\begin{tabular}{|c|cccccccccccc||c|}
\hline
\scalebox{0.7}{n}&\scalebox{0.7}{$0$}&\scalebox{0.7}{$1$}&\scalebox{0.7}{$2$}&\scalebox{0.7}{$3$}&\scalebox{0.8}{$4$}&\scalebox{0.7}{$5$}&\scalebox{0.7}{$6$}
&\scalebox{0.7}{$7$}&\scalebox{0.7}{$8$}&\scalebox{0.7}{$9$}&\scalebox{0.7}{$10$}&\scalebox{0.7}{$11$}&\scalebox{0.7}{$GW_3/\NN_3$}\\
\hline
\scalebox{0.70}{$GW_{3,n}/\NN_{3,n}$}&
\scalebox{0.7}{$0/0$}&\scalebox{0.7}{64/0}&\scalebox{0.7}{1664/8}&\scalebox{0.7}{12864/0}&\scalebox{0.7}{47328/40}
&\scalebox{0.7}{98688/0}
&\scalebox{0.7}{125184/64}
&\scalebox{0.7}{98688/0}
&\scalebox{0.7}{47328/40}
&\scalebox{0.7}{12864/0}&\scalebox{0.7}{1664/8}&\scalebox{0.7}{64/0}
&\scalebox{0.7}{446400/160}\\
\hline
\end{tabular}
}
\end{table}


\begin{remark}
It turns out that the values of $GW_{m,1}$ and $GW_{m,1}^{[2]}$
do not depend on $m\ge 1$, but only on the number $s\ge2$ of singular fibers in a conic bundle $X$.
Namely,
$$GW_{m,1}=2^{s-1},\qquad GW_{m,1}^{[2]}=2^{s-2},$$
which follows from the relations in Theorems \ref{simple-recursion-any-d} and \ref{moment-recursion-any-d}, or alternatively, can be derived directly,
using a simple observation that all curves from layers $(m,1)$ have arithmetic genus 0.

Using Theorem \ref{magic-theorem} we conclude that
$$\NN_{2m,2}=2^{m+s-5},\quad\text{ for } s\ge4.$$
\end{remark}

\subsection{
Rational surfaces admitting no real conic bundle structure}\label{except}

\begin{proposition} {\rm (}{\it cf. }\cite[Theorem 5.7]{Manin-II}{\rm )}\label{exceptional-cases} 
Real del Pezzo surfaces $X$ with $K_X^2=1$, $X_\R=\R P^2\+4S^2$ and $Y$ with
$K_Y^2=2$, $Y_\R=4S^2$ do not admit a real conic bundle structure.
 Moreover, if surfaces $X'$ and $Y'$ are birational over the reals respectively to $X$ and $Y$, then 
 they also do not
admit a real conic bundle structure.
\end{proposition}

Here is a proof for convenience of the reader.

\begin{proof}
In the first case, we observe that for any $X'$ birational over the reals to $X$ the conj-invariant sublattice of $H_2(X')$ is $E_8\oplus rA_1$, while
for real conic bundles real birational to a minimal spherical real conic bundle with 5 connected components, such sublattice is $D_{10}\oplus sA_1$.
So, the result follows from the uniqueness of decomposition into irreducible summands
for elliptic root lattices.

In the second case, the conj-invariant sublattice of any $Y'$ birational over the reals to $Y$ is $E_7\oplus rA_1$ while on the conic bundle side it is $D_8\oplus sA_1$, and the same argument applies.
\end{proof}

It is easy to notice (e.g., by comparing the above proof with the classification of minimal models) that the class of surfaces described in Proposition \ref{exceptional-cases}
is closed under deformations.

It is known also (see \cite{iskov}) that each of the other real rational surfaces with $K^2\le 6$ admits a real rational ruling
and hence admits a real conic bundle structure after a small real deformation of the surface. Thus,
the surfaces described in Proposition \ref{exceptional-cases} are the only ones to which
our counting scheme can not be applied directly, even after a deformation.
However, every 
of these exceptional surfaces can be connected by wall-crossing with a surface admitting a real conic bundle. Therefore, it seems natural to expect that
our bi-graded counting can be extended to these two exceptional classes of surfaces as well. Then, for example, in the case of  $K_X^2=1$, $X_\R=\R P^2\+4S^2$ the number 
$N_1=W(-K)=8$ should split as $2+4+2$ and $N_2=W(-2K)=30$ as $1+8+12+8+1$, see Table \ref{k=8}, and a question arises: what can be a geometrical meaning of such a splitting?

\subsection{Rule of alternation}
There are 3  topological models for a neighborhood of a real fiber in
a real embedded conic bundle $X\subset\P(\E)\to \P^1$:
\begin{itemize}\item
$x^2+y^2+tz^2=0\subset \P^2_\R\times (-1,1)$ near
an {\it elliptic singular fiber} (by definition, the irreducible components of such a fiber are complex-conjugate lines in $\P^2$),
\item
$x^2-y^2+tz^2=0 \subset \P^2_\R\times (-1,1)$ near a {\it hyperbolic singular fiber}
(by definition, its components are real lines),
\item $x^2+y^2\pm z^2=0\subset \P^2_\R\times (-1,1)$ near a nonsingular fiber (empty, or nonempty).
\end{itemize}

For a real conic bundle $X\to \P^1$, the consecutive order of its hyperbolic and elliptic fibers along a directed interval
$[a_1,a_2]\subset \P^1_\R$
will be encoded by a word formed with ``$h$'' and ``$e$'', and called the {\it singular code} of $X$ along $[a_1,a_2]$.

The well-known fact that the mapping class group $\Map(\P^2_\R)$,
is trivial, while 
$\Map(\P^2_\R, S^1)=\Z/2$, 
 implies the following uniqueness result.
\begin{proposition} 
Assume that real embedded conic bundles $X\subset\P(\E^X)\xrightarrow{\pi_X} \P^1$ and $Y\subset\P(\E^Y)\xrightarrow{\pi_Y} \P^1$
have the same singular codes along $[a_1,a_2]$ and $[b_1,b_2]$ respectively and that non-singular fibers  $X_\R\cap\pi_X^{-1}(a_1)$ and 
$Y_\R\cap\pi_Y^{-1}(b_1)$
are both empty or non-empty. 
Then there exists a fiberwise diffeomorphism
$ \P_\R(\E^X)\cap\pi_X^{-1}[a_1,a_2] \to\P_\R(\E^Y)\cap\pi_Y^{-1}[b_1,b_2]$
sending $X_\R\cap \pi_X^{-1}[a_1,a_2]$ to $Y_\R\cap \pi_Y^{-1}[b_1,b_2]$.
\qed\end{proposition}

As an easy consequence, this implies the following rule.

\begin{cor}\label{rule}
Let $X\subset\P(\E)\to \P^1$ be a real embedded conic bundle, $[a_1,a_2]\subset \P_\R^1$ an interval, and $\theta$ a $\Pin^-$-structure on $X_\R\cap \pi_X^{-1}[a_1,a_2]$ inherited
from  a $\Pin^-$-structure on $\P(\E_\R)\cap \pi_\E^{-1}[a_1,a_2]$ by fixing a half of $\P(\E_\R)\cap \pi_\E^{-1}[a_1,a_2]$ bounding $X_\R\cap \pi_X^{-1}[a_1,a_2]$ (see Sections \ref{inh-even}, \ref{inh-odd}).
Assume that $\pi^{-1}(a_i)=L_i'\cup L_i''$, $i=1,2$, are hyperbolic real singular fibers. Then, the value $q_\theta(L_1')=q_\theta(L_1'')$ coincides with
the value $q_\theta(L_2')=q_\theta(L_2'')$, if and only if the number of hyperbolic singular fibers over $[a_1, a_2]$ is odd.
\qed\end{cor}

If the number of real singular fibers is odd, and $C\subset X$ is a fixed nonsingular real fiber, then this rule of alternation determines a unique up to sign inherited $\Pin^-$-structure on $X_\R\setminus C_\R$ ({\it cf.} Section \ref{inh-odd}).
In the case of even number of real singular fibers, this rule determines an inherited $\Pin^-$-structure on $X_\R$  up to sign and shift by class of the fibers. Applying this rule one can easily find, say, among multiple blowups
of the real projective plane, examples where the same real rational surface has several conic bundles providing different sets of inherited $\Pin^-$-structures.

{\bf Example.} Consider a del Pezzo surface $X$ obtained by blowing up $\P^2$ at points $p_0,\dots, p_4\in\P^2_\R$ and denote by
$E^i \subset X$ the exceptional curve over $p_i$.
In some affine chart the points $p_i$ form a convex polygon and we choose their consecutive cyclic numeration.
Consider conic bundle
$\pi_i \:X\to\P^1$ obtained by lifting of the pencil of lines in $\P^2$ centred at $p_i$ and let 
$q_i\: H_1(X_\R)\to\Z/2$ be 
the quadratic function associated to any of the four inherited $\Pin^-$-structures associated with  $\pi_i$.

Then the alternation rule implies that
for each $i$, the sequence of values $q_i(E^{i+k}_\R)$, $k=1,\dots,4$ is alternating (here we add $i+k$ in $\Z/5$).
In particular, $q_0(E^{1}_\R)=q_0(E^3_\R)$, while $q_2(E^{1}_\R)\ne q_2(E^{3}_\R)$, which implies that the four $\Pin^-$-structures inherited from 
$\pi_0$ and the four inherited from 
$\pi_2$
are disjoint.

\end{document}